\documentclass{article}
  \usepackage{amsmath}
\usepackage{graphicx}
\usepackage{hyperref}
\usepackage{amsthm,amsfonts}
\usepackage{color}
\usepackage{enumerate}

\newtheorem{theorem}{Theorem}[section]
\newtheorem{corollary}{Corollary}

\newtheorem{lemma}[theorem]{Lemma}
\newtheorem{proposition}{Proposition}
\theoremstyle{definition}
\newtheorem{definition}[theorem]{Definition}
\newtheorem{remark}{Remark}
\newtheorem*{notation}{Notation}
\newcommand{\ep}{\varepsilon}

\newcommand{\R}{\mathbb{R}}
\newcommand{\N}{\mathbb{N}}
\newcommand{\Rn}{\mathbb{R}^N}
\newcommand{\CC}{\mathrm{C}}
\newcommand{\HH}{\mathcal{H}}
\newcommand{\la}{\lambda_\alpha}
\newcommand{\pa}{\varphi_\alpha}
\newcommand{\supp}{\mathop{\mathrm{supp}}\nolimits}
\newcommand{\us}{\overline{u}}

\newcommand{\vm}{\overline{v}}
\newcommand{\vt}{\tilde{v}}
\newcommand{\cu}{c^*}
\newcommand{\psm}{\underline{V}}
\newcommand{\tw}{\tilde{w}}
\newcommand{\tf}{\tilde{f}}
\newcommand{\tV}{\tilde{V}}

\title{Traveling fronts guided by the environment for reaction-diffusion equations}

\author{}

\date{}

\begin{document}

\maketitle

\centerline{\scshape Henri Berestycki }
\medskip
{\footnotesize
 \centerline{CAMS, UMR 8557, EHESS}
   \centerline{190-198 avenue de France, 75244 Paris Cedex 13, France}
} 

\medskip

\centerline{\scshape Guillemette Chapuisat }
\medskip
{\footnotesize
 \centerline{LATP, UMR 7353, Aix-Marseille Universit\'e}
   \centerline{39 rue F. Joliot-Curie, 13453 Marseille Cedex 13, France}
 \centerline{CAMS, UMR 8557, EHESS}
   \centerline{190-198 avenue de France, 75244 Paris Cedex 13, France}
} %
\bigskip

\centerline{Dedicated to Hiroshi Matano on the occasion of his {\em Kanreki}.}

\begin{abstract}
This paper deals with the existence of traveling fronts for the 
reaction-diffusion equation:
$$
\frac{\partial u}{\partial t} - \Delta u =h(u,y)  \qquad t\in \R, \; x=(x_1,y)\in \R^N.
$$
We first consider the case $h(u,y)=f(u)-\alpha g(y)u$ where $f$ is of KPP or bistable type and $\lim_{|y|\rightarrow +\infty}g(y)=+\infty$. This equation comes from a model in population dynamics in which there is spatial spreading as well as phenotypic mutation of a quantitative phenotypic trait that has a locally preferred value. The goal is to understand spreading and invasions in this heterogeneous context. We prove the existence of threshold value $\alpha_0$ and of a nonzero asymptotic profile (a stationary limiting solution) $V(y)$ if and only if $\alpha<\alpha_0$. When this condition is met, we prove the existence of a traveling front. This allows us to completely identify the behavior of the solution of the parabolic problem in the KPP case.

We also study here the case where $h(y,u)=f(u)$ for $|y|\leq L_1$ and $h(y,u) \approx - \alpha u$ for $|y|>L_2\geq L_1$. This equation provides a general framework for a model of cortical spreading depressions in the brain. We prove the existence of traveling front if $L_1$ is large enough and the non-existence if $L_2$ is too small.
\end{abstract}

\bigskip

\section{Introduction}
\label{section: intro}

This paper deals with the existence of bounded traveling fronts for the  
reaction-diffusion equation 
\begin{equation}
\label{main}
\frac{\partial u}{\partial t} - \Delta u =h(y,u) \qquad t\in \R, \; x=(x_1,y)\in \R^N.
\end{equation} 
The function $h$ will be of three different forms in this paper. The first two concern non-linear terms  $h(y,u)=f(u) - \alpha g(y) u$ where $f:\R \rightarrow \R$ is $\CC^1$, and is either of positive type, or of bistable type and $g:\R^{N-1} \rightarrow \R_+$ is $\CC^0$, $g(0)=0$ and $g\xrightarrow{|y|\rightarrow +\infty} +\infty$. 
The existence of traveling front depends on the value of $\alpha>0$. The third case we consider here is when $h(y,u)=f(u)$ for $|y|\leq L_1$ and $ h(y,u)\leq -mu$ for $|y|\geq L_2$ where $0<L_1\leq L_2 <\infty$ are given parameters and $f$ is of bistable form and $h(y,u)+mu \rightarrow 0$ for $|y|\rightarrow +\infty$. We study the existence of traveling fronts depending on the value of $L_1$ and $L_2$.\\

The problems we study in this paper bear some similarities with the question of traveling fronts in cylinders of \cite{BN}. However there are important differences that have to do with the fact that the cross section in \cite{BN} was bounded and only the Neumann condition was considered there. Whereas here, the problem is posed in the whole space and the solution vanish at infinity in directions orthogonal to the direction of propagation.  We follow the same general scheme as in \cite{BN} and in particular make use of the sliding method. But some new ideas are also required. In particular, first, we treat directly the KPP case without the approximation of the KPP non-linearity by a combustion non-linearity as in \cite{BN}. Then in the approach of Berestycki - Nirenberg \cite{BN} to traveling fronts in cylinders for the bistable case, a useful result of H. Matano \cite{HM79} was involved in the proof. Here, we rely on stability ideas but also use energy minimization properties to bound the speed of the solution in the finite domain approximation.  In particular, we do not use the precise exponential behavior that was used in \cite{BN}.
Actually the developments of this method that we present in this paper can be used to somewhat simplify parts of \cite{BN}. They  can also be applied to traveling fronts in cylinder with Robin or Dirichlet boundary conditions. \footnote{The construction of traveling fronts for Neumann and Dirichlet conditions in cylinders given by \cite{Vega93}  appears to be incomplete. Indeed, the continuity of the function $\phi$ on page 515 is not established so that using Dini's Theorem to derive Lemma 3.2 there is not justified.}\\

Equation \eqref{main} in the first case comes from a model in population dynamics \cite{DFP03} that we briefly describe now. Let $u(t,x,v)$ represent the density of individuals at time $t$ and position $x$ that possess some given quantitative phenotypic trait represented by a continuous variable $v\in \R$. 
For example, the latter could be the size of wings or the height of an individual. We assume that individuals follow a brownian motion (i.e. they diffuse) in space with a constant diffusion coefficient $\nu$, reproduce identically and disappear with a growth rate $k(x,v)$ that depends on the position $x$ and on the trait $v$. Furthermore, they also reproduce with mutation that is represented by a kernel $K(x,v,w)$ and disappear due to competition with a constant $L>0$. Thus, one is led to the following equation for $u$:  
\begin{equation}
\label{DP1}
\begin{array}{l}
 \partial_t u(t,x,v) - \nu \Delta_x u(t,x,v) = 
  k(x,v)\, u(t,x,v) \\ \qquad \qquad \qquad 
 + \int_{w}  K(x,v,w) \,u(t,x,w)\,dw 
- u(t,x,v)\, \int_{w}
 L u(t,x,w)\, dw .  
\end{array}
\end{equation}
We assume moreover that there exists a most adapted trait $\phi=\phi(x)$ that may depends on the location $x$. The farther the trait of an individual is from the most adapted trait, the larger the probability of dying and not reproducing.  Thus the growth rate can be written for example as $k(x,v) = a - b\,|v - \phi(x)|^2$ with $a$ and $b>0$.  

Non-local reaction-diffusion equations of this type raise some new difficulties from a mathematical standpoint as shown in \cite{BNPR09}. There, behaviors that are quite different from those in local equations are brought to light. After this paper was completed, we learned that in \cite{ACR} the existence of traveling front was also derived for equation
$$
 \partial_t u(t,x,v) - \nu \Delta_{x,v} u(t,x,v) = 
  k(x,v)\, u(t,x,v)
- u(t,x,v)\, \int_{w}
 L(v,w) u(t,x,w)\, dw . 
$$
This work follows in part the methods of \cite{BNPR09}. As in \cite{BNPR09}, the nonzero limiting stationary state is not prescribed.
 In a forthcoming numerical study \cite{avenir_num}, we study the full equation \eqref{DP1} and we discuss the monotonicity of fronts depending on the value of $a$ and $b$. 

In this paper, we introduce a simplified version of this model that emphasizes propagation guided by the environment. First, we assume that mutations are due to a diffusion process represented by a Brownian motion in the space of trait $v$. Furthermore, we assume that $\phi$ is linear. Then a rotation in the variables $(x_1,y)$ allows one to reduce the problem to the case where the most adapted trait is $y=0$. Therefore we assume $\phi(x)=0$ and \eqref{DP1} can be rewritten as 
\begin{equation}
\label{DP2}
 \partial_t u(t,x,v) - \nu \Delta_{x,v} u(t,x,v) =  (a - b|v|^2)u(t,x,v) 
- u(t,x,v)\, \int_{w} L u(t,x,w)\, dw .
\end{equation}
Lastly we assume that competition is only between individuals sharing the same trait which leads us to equation
\begin{equation}
\label{DP3}
 \partial_t u - \nu \Delta_{x,v}u =  (a -Lu)u - b|v|^2u. 
\end{equation}
Equation \eqref{main} is a  generalization of this equation.  In \cite{DFP03}, the authors observe numerically a generalized transition front spreading along the graph of $\phi$ for equation \eqref{DP1} (see \cite{BH06, BH07, BH12, HM12} or \cite{Shen} for the definition of generalized transition fronts). Here we want to prove theoretically (i) that there exists such a front for equation \eqref{main} at least for some values of the parameter $\alpha>0$ and (ii) that extinction occurs if $\alpha$ is too large. The latter condition can be interpreted as saying that the ``area'' of adapted traits is too thin compared to the diffusion. To remain consistent with the biological motivation, we only consider here non-negative and bounded solutions of \eqref{main}.\\

Other types of models related to this one have been proposed in the literature. For example, the model developed by Kirkpatrick and Barton in 1997 \cite{KB97} also studies the evolution of a population and of its mean trait. The main difference is that they have a system in $u$ and $v$ where $u$ represents the population and $v$ the mean trait is described by a specific equation. This model has been further explored00 \cite{FHB08, HBFF11}. It is worth noting that these models use the same type of non-linearity for the adaptation to the environment and model the mutation with the Laplace operator as well rather than integral operators.\\

This type of reaction-diffusion process in heterogeneous media also arises in many contexts in medicine. An important class of such models was treated in \cite{C07, PMHC09}. They deal with the propagation of a {\em cortical spreading depression\/} (CSD) in the human brain. These CSD's are transient depolarizations of the brain that slowly propagate in the cortex of several animal species after a stroke, a head injury, or seizures \cite{somjen}. They also are suspected of being responsible for the aura in \textit{migraines with aura}. CSD's are the subject of intensive research in biology since experiments blocking them during strokes in rodents have produced very promising results \cite{DeKeyser99, Nedergaard95}. These observations however have not been confirmed in humans and the existence of CSD's in the human brain is still a matter of debate \cite{Mayevsky96, Gorgi01, Back00, Strong02}. Since very few experiments and measurements on human brain are available be it for obvious ethical or technical reasons, mathematical models of a CSD is helpful in understanding their existence and conditions for their propagation. In such a problem, the morphology of the brain and thus the geometry of the domain where CSD's propagate, is believed to play an important role. \\

The brain is composed of gray matter where neuron's soma are and of the white matter where only axons are to be found. The rodent brain (on which many of the biological experiments are done) is rather smooth and composed almost entirely of gray matter. On the opposite, the human brain is very tortuous. The gray matter is a thin layer at the periphery of the brain with much thickness variations and convolutions, the rest of the brain being composed of white matter. According to mathematical models of CSDs \cite{Cetal08, somjen, shapiro01, tuckwell80}, the depolarization amplitude follows a reaction-diffusion process of bistable type in the gray matter of the brain while it diffuses and is absorbed in the white matter of the brain. The modeling of CSD hence leads one to the study of equations of the following type:
\begin{equation}
\label{SD}
\frac{\partial u}{\partial t}-\Delta u =f(u) \mathbf{1}_{|y|<L}-\alpha u\mathbf{1}_{|y|\geq L} \qquad t\in \R, \; x=(x_1,y)\in \R^N.
\end{equation} 
Here, $f$ is of bistable type and $|y|=L$ corresponds to the transition from gray matter to white matter. This equation is of type \eqref{main} and we also study it here in sections \ref{section: APbist} and \ref{section: TFbist} where we extend earlier works on the subject. In \cite{C07}, this equation was studied to prove that the thinness of the human gray matter ($L$ small) may prevent the creation or the propagation of CSDs on large distances. It was proved by studying the energy in a traveling referential of the solution of \eqref{SD} with a specific initial condition. The special case of \eqref{SD} for $N=2$ was described more completely in \cite{CJ11}. In \cite{PMHC09}, a numerical study shows that the convolutions of the brain have also a strong influence on the propagation of CSD. In \cite{CG05}, the effect of rapid variations of thickness of the gray matter was studied. \\

Lastly, let us note that the same kind of equation arises in the modeling of tumor cords but with a slightly more complicate KPP non-linearity. We plan to investigate this model in our forthcoming work \cite{avenir_cancer}.\\

As already mentioned, the study of propagation of fronts and spreading properties in heterogeneous media is of intense current interest. For instance, the existence of fronts propagating in non-homogeneous geometries with obstacles has been established in Berestycki, Hamel and Matano \cite{BHM09}. Definitions of generalized waves have been given by Berestycki and Hamel in \cite{BH07} and \cite{BH12} where they are called generalized transition waves. Somewhat different approaches to generalizing the notions of traveling fronts have been proposed by H. Matano \cite{HM12} and W. Shen \cite{Shen}.
The existence of fronts for non-homogeneous equations are established in \cite{NRRZ12} and \cite{Z12}. \\

Let us first introduce some notations before stating the main results.
\begin{notation}
We note $x=(x_1,y) \in \R^N$ where $x_1\in \R$ and $y\in
\R^{N-1}$. Hence $x$ is the space variable in $\R^N$, $x_1$ is its first coordinate and $y$ is the vector of $\R^{N-1}$ composed of all the other coordinates of $x$. 
As usual $B_R=B(0,R)$ denotes a ball of radius $R$ centered at 0, but here it will always mean the ball in $\R^{N-1}$.  
\end{notation}
First we are interested in solutions of 
\begin{equation}
\label{EP}
\begin{cases}
\displaystyle \frac{\partial u}{\partial t}- \Delta u = f(u)-\alpha g(y) u, \quad
x=(x_1,y)\in \R^N, \; t\in \R \\
u\geq 0, \quad u \text{ bounded,}
\end{cases}
\end{equation}
with $\alpha>0$. We will assume that $f:\R \rightarrow \R \text{ is } \CC^1$ and satisfies either one of the following conditions:
$$
f(0)=f(1)=0,\; f>0 \text{ on } (0,1) \text{ and }  f'(0)>0,
$$
or
$$
\begin{array}{c}
 \text{ there exists } \theta \in ]0,1[ \text{ such that } f(0)=f(\theta)=f(1)=0,\\
 f<0 \text{ on } (0,\theta) \text{ and } f>0 \text{ on } (\theta,1) 
\text{ with } f'(0)<0, \; f'(1)<0, \\
\text{and } \int_0^1 f(s)ds>0.
\end{array}
$$
The first case will be referred to as the positive case and the second one will be called bistable case. Furthermore, if $f$ is in the positive case and if 
$$
s\mapsto \frac{f(s)}s \text{ is decreasing on }(0,1]
$$
we will say that $f$ is of Fisher-KPP type.
Since we are only interested in solutions of \eqref{EP} in $[0,1]$, we will further assume that $f(s)\leq 0$ for $s\geq 1$. 
Moreover we assume
\begin{equation}
\label{Hypgpos}
 g:\R^{N-1} \rightarrow \R_+ \text{ is continuous, } g(0)=0, \; g>0 \text{ on }\R^{N-1}\setminus\{0\} 
\end{equation}
(except in section \ref{gvanish} where $g$ can vanish) and
\begin{equation}
\label{Hypginf}
 \lim_{|y|\rightarrow +\infty}g(y)=+\infty. 
\end{equation}
 Taking $g(y)=|y|^2$ and $f(s)=as(1-s)$ yields the particular case of equation \eqref{DP3}.\\

This paper is concerned with the long term behavior of \eqref{EP} and with the existence of curved traveling fronts, i.e. solutions $u(t,x)=U(x_1-ct,y)$ with $c\in \R$ a constant and $U:\R^N \rightarrow \R$ such that the limits $\lim_{s\rightarrow \pm \infty}U(s,.)$ exist uniformly and are not equal. Regarding these fronts, our main results are the following.
\begin{theorem}
\label{existFP}
If $f$ is of Fisher-KPP type, there exists $\alpha_0>0$ such that: 
\begin{itemize}
\item For $\alpha \geq \alpha_0$, there exists no traveling front solution of \eqref{EP},
\item For $\alpha <\alpha_0$ there exists a threshold $c_*>0$ such that there exists a traveling front of speed $c$ of equation \eqref{EP} if and only if $c\geq c_*$. 
\end{itemize}
\end{theorem}
This existence theorem gives us information on the behavior of the solution of the parabolic problem. In this paper we prove the following theorem:
\begin{theorem}
If $f$ is of Fisher-KPP type, for $u_0\in L^\infty$, there exists a unique solution $u(t,x)$ of
$$
\begin{cases}
\partial_t u -\Delta u =f(u) -\alpha g(y) u \quad
&\text{on } (0,+\infty)\times \R^N,\\
u(0,x)=u_0(x) & \text{on } \R^N.
\end{cases}
$$
\begin{itemize}
\item If $\alpha\geq \alpha_0$, it verifies $u(t,x) \xrightarrow{t\rightarrow
  +\infty} 0$ uniformly with respect to $x\in \R^N$.
\item If $\alpha < \alpha_0$ and $u_0\in \CC^0(\R^N)$ is compactly supported with $u_0< V$ where $V=V(y)$ is the unique positive asymptotic profile (stationary solution), then 
\begin{eqnarray*}
\text{for any }c>c^* \quad \lim_{t\rightarrow +\infty} \sup_{|x_1|\geq ct}
u(t,x)=0, \\
\text{for any } c \text{ with } 0\leq c <c^* \quad \lim_{t\rightarrow +\infty} \sup_{|x_1| < ct} |u(t,x)-V(y)|=0.
\end{eqnarray*}
\end{itemize}
\end{theorem}

This means that there is a threshold value $\alpha_0$ such that for $\alpha \geq \alpha_0$, there is {\em extinction}. On the contrary, when $\alpha \leq \alpha_0$, there is spreading and the state $V(y)$ invades the whole space. The asymptotic speed of spreading is then $c^*$.
The property of asymptotic spreading is in the same spirit of the theorem of asymptotic speed of spreading in cylinders established by Mallordy and Roquejoffre in \cite{MR95}.

Theorem 1.2 has interesting consequences for the dynamics of the phenotypic diversity in a population. Several studies have tried to understand population migrations through phenotypic diversity \cite{Excoffier09, Hallatschek07, Hallatschek08, Hallatschek10, Roques12, Vlad04}. Our invasion result states that for large times, one expect to see the state $V(y)$ at any location (and not the migration process) and it holds whatever the initial distribution of the population is. Note furthermore that the profile $V(y)$ is unique. Hence whatever the initial structure of the population is, the phenotypic diversity at large times is completely determined by the profile of the function $g$. \\

In the slightly more general case of a positive non-linearity, we will prove the following existence theorem.
\begin{theorem}
\label{existpos}
If $f$ is of positive type, there exists $\alpha_0>0$ such that for $\alpha <\alpha_0$ there exists a traveling front of equation \eqref{EP}. 
\end{theorem}

Regarding the case of bistable $f$ we have the following result:
\begin{theorem}
\label{existFPb}
If $f$ is of bistable type, there exist $\alpha^* \geq \alpha_*>0$ such that 
\begin{itemize}
\item For $\alpha\geq \alpha^*$, there exists no traveling front solution of \eqref{EP},
\item  For $\alpha <\alpha_*$, under condition~\ref{uniqsPA} of Section~\ref{section: APbist}, there exists a traveling front $u$ of speed $c>0$ solution of \eqref{EP}. 
\end{itemize}
\end{theorem}

Lastly, the model for CSD's leads one to equations of the type
\begin{equation}
\label{Egenintro}
\partial_t u -\Delta u=h(y,u) \quad x=(x_1,y)\in \R^N.
\end{equation}
where $h(y,u)$ verifies
\begin{eqnarray*}
&& h(y,u)=f(u) \text{ for } |y|\leq L_1 \\
&& h(y,u)\leq -mu \text{ for } |y|\geq L_2\\
&& h(y,u)+mu\xrightarrow{|y|\rightarrow +\infty} 0 \quad \text{uniformly for } u\in \R^+
\end{eqnarray*}
where $0<L_1\leq L_2 <\infty$ and $m>0$ are given parameters and $f$ is of bistable form. \\

In this paper we prove the following Theorem.
\begin{theorem}
\label{thmSDTF}
There exist critical radii $0<L_*\leq L^*<\infty$ with the following properties:
\begin{itemize}
\item For $L_2<L_*$, there is no traveling front solution of \eqref{Egenintro}.
\item For $L_1>L^*$ (independently of $L_2$), assuming that there is a unique stable asymptotic profile of \eqref{APSD}, there exists a traveling front of speed $c>0$ solution of \eqref{Egenintro}.
\end{itemize}
\end{theorem}
The assumption on the uniqueness of the asymptotic profile is proved to be true for the case $N-1=2$, $L_1=L_2$ and $h(y,s)=-ms$ for $|y|\geq L_2$. This is done in \cite{CJ11} by phase plane method. For want of a uniqueness result for the profile equation in more general cases,  

 This theorem completes the study in \cite{C07} on the existence of CSD in the human brain. Indeed in \cite{C07} the transition from gray to white matter was instantaneous when biologically there is a smooth transition from
 gray to white matter. This Theorem confirms the intuition that CSD's can be found in part of the human brain where the gray matter is sufficiently thick but they can not propagate over large distances due to a thin gray matter in many parts of the human brain.\\
 
The paper is organized as follows. In section \ref{section: prelim} we state some preliminary results that will be used in the sequel. Section \ref{section: asympt} is dedicated to the study of the existence and uniqueness of non-zero asymptotic profiles for a traveling front solution of \eqref{EP}. In section \ref{section: convergence} we study the large time behavior. There we prove extinction if $\alpha \geq \alpha_0$ and convergence towards the front of minimal speed if $\alpha <\alpha_0$. Section \ref{section: positive} extends existence of traveling front results to the case of a positive non-linearity. Then, section \ref{section: APbist} is devoted to the study of the asymptotic profiles in the bistable case and section \ref{section: TFbist} to the existence of traveling front for $\alpha<\alpha_*$ in the bistable case. Lastly, in section~\ref{section: CSD} we describe the precise problem arising in the modeling of CSD's and state our main result in this framework.


\section{Preliminary results}
\label{section: prelim}

In our proofs, we will need several times the exponential decay of the asymptotic profile which can be easily proved from the following theorem established in \cite{BR08}. 
\begin{theorem}
\label{BR}
Let $v\in H^2_{\text{loc}}(\R^N)$ be a positive function. Assume that
there exists $\gamma>0$ and $C>0$ such that 
$$
\forall x\in \R^N, \quad v(x)\leq Ce^{\sqrt{\gamma}|x|} \text{ and }
\liminf_{|x|\rightarrow \infty} \frac{\Delta v(x)}{v(x)}>\gamma.
$$
Then, $\displaystyle \lim_{|x|\rightarrow
  \infty}v(x)e^{\sqrt{\gamma}|x|}=0$.
\end{theorem}
This result is established in \cite{BR08}, lemma 2.2.
In the context of equation \eqref{main}, we thus have the following corollary.
\begin{corollary}
\label{decry}
Let $u$ be a non-negative and bounded solution of 
$$
\Delta v +f(v)-\alpha g v=0 \quad \text{on } \R^{N-1}.
$$
Then, for any $\gamma>0$ there exists $C>0$ such that 
$$
0\leq v(y) \leq Ce^{-\gamma |y|} \qquad \text{and} \qquad
|\nabla v(y)|\leq Ce^{-\gamma |y|}.
$$
\end{corollary}
\begin{proof}
The estimate on $v$ comes directly from Theorem \ref{BR} and the estimate on $|\nabla v|$ derives from standard global $L^p$ estimates. 
\end{proof}

\section{The case of a  Fisher-KPP non-linearity. Asymptotic profiles.}
\label{section: asympt}

In this section, we are interested in the asymptotic profiles of a
traveling front solution of \eqref{EP} as $x_1\rightarrow \pm \infty$. Hence, we are looking for solutions of the following equation  
\begin{equation}
\label{PA}
\begin{cases}
\Delta V +f (V) -\alpha g(y) V=0, \quad & y\in \R^{N-1},
\\
V \geq 0, \quad V \text{ bounded.}
\end{cases}
\end{equation}
We assume that $f:\R \rightarrow \R \text{ is } \CC^1$,
\begin{equation}
\label{Hypfpos}
 f(0)=f(1)=0,\; f>0 \text{ on } (0,1)
\end{equation}
and 
\begin{equation}
\label{HypfFisher-KPP}
 s\in (0,1]\mapsto \frac{f(s)}{s} \text{ is decreasing.}
\end{equation}
Since the constant function 0 is always a solution, the problem
is to know when there exist non-zero solutions. As we will see here, the existence of such a positive asymptotic profile is characterized by the
sign of the principal eigenvalue of the linearized operator around 0. We now make this notion precise.
 
\subsection{Principal eigenvalue of the linearized operator}
\label{principal}

To start with, let us define the natural weighted space $$\HH=\{v\in
H^1(\R^{N-1}) \, , \, \sqrt{g}u\in 
L^2(\R^{N-1})\}$$ and its associated norm. For $v\in \HH$, we set $\|v\|_\HH = (
\|v\|^2_{H^1}+\|\sqrt{g}v\|^2_{L^2})^{\frac{1}{2}}$. 
The linearized operator about 0 is $L\varphi=-\Delta \varphi
+\big(\alpha g(y)-f'(0)\big)\varphi$ for $\varphi \in \HH$. We are interested in the
eigenvalues of $L$. Even though the problem is set on all of $\R^{N-1}$, the term
in $\alpha g(y)$ yields compactness of
the injection $\HH \hookrightarrow L^2(\R^{N-1})$. Hence the existence
of a principal eigenvalue is obtained as usual.
\begin{theorem} Let us define 
$$
R_\alpha(\varphi)= \frac{\int |\nabla \varphi|^2+
  \big(\alpha g-f'(0)\big)\varphi^2 }{\int \varphi ^2}.
$$

The operator $L$ has a smallest eigenvalue 
\begin{equation}
\la= \inf_{\varphi \in \HH \setminus \{0\} } R_\alpha (\varphi).
\label{QR}
\end{equation}
Moreover there exists a unique positive eigenfunction associated
  with $\la$ of $L^2$-norm equal to 1, called $\pa$ in the following.
The eigenspace associated with $\la$ is spanned by  $\pa$.
\end{theorem}
The proof is classical due to the compactness of $\HH
\hookrightarrow L^2(R^{N-1})$. We refer for example to \cite{Evans}.   
\begin{remark}
If $g(y)=|y|^2$, the problem can be rescaled  and we obtain the harmonic
oscillator for which principal eigenvalue and eigenfunction are well
known \cite{schwartz}. In that case, $\la=(N-1)\sqrt{\alpha}-f'(0)$ and
$\pa=\left( \frac{\sqrt{\alpha}}{\pi}\right)^{\frac{1}{N-1}}
e^{-\frac{\sqrt{\alpha}}{2}|y|^2}$.  
\end{remark}
Since the existence of a positive solution of \eqref{PA} will depend on the
sign of the principal eigenvalue, the following proposition describes the
behavior of $\la$ as a function of $\alpha$.
\begin{proposition}
\label{lacomp}
The function $\alpha \mapsto \la$ is continuous, increasing and concave
for $\alpha \in (0, +\infty)$. Moreover $\lim_{\alpha \rightarrow 0}
      \la =-f'(0)$ and for $\alpha$ large enough $\la>0$. 
\end{proposition}
\begin{proof}
Let us fix $\alpha >0$ and $\eta>0$.
Equation \eqref{QR} shows that
$$
\lambda_{\alpha + \eta} \leq \int |\nabla \pa|^2 +\big(
(\alpha+\eta)g-f'(0)\big) \pa^2 =\la +\eta \int g(y) \pa^2.
$$
Similarly, we obtain $\la \leq \lambda_{\alpha +\eta} -\eta \int g
\varphi_{\alpha+\eta}^2$. From this we derive: 
$$
0< \eta \int g \varphi_{\alpha+\eta}^2 \leq \lambda_{\alpha
  +\eta}-\la \leq \eta \int g\pa^2.
$$
This and similar computation for $\la - \lambda_{\alpha-\eta}$ yields that $\alpha \mapsto \la$ is increasing and locally Lipschitz on $(0,+\infty)$.

Concavity is classical. It suffices to observe that for each fixed $\varphi$, $$
\alpha \mapsto R_\alpha(\varphi)= \frac{\int |\nabla \varphi|^2+
  (\alpha g-f'(0))\varphi^2 }{\int \varphi ^2}
$$ 
is an affine function of $\alpha$ and that $\la= \inf_{\varphi \in \HH \setminus \{0\} } R_\alpha (\varphi)$.

In order to prove that $\la \xrightarrow{\alpha \rightarrow 0}-f'(0)$,
for any $\ep >0$ choose a function $\psi_\ep
$ of compact support with
$\|\psi_\ep\|_{L^2}=1$ and $\int |\nabla \psi_\ep|^2< \ep$. Let $\supp \psi_\ep \subset B_{R_\ep}$.
From \eqref{QR} we get  
\begin{equation*}
-f'(0) \leq \la \leq \ep+ \alpha \max_{B_{R_\epsilon}}g-f'(0) 
\end{equation*}
So for any $\displaystyle \alpha <\frac{\ep}{\max_{B_{R_\epsilon}}g}$, 
$$
-f'(0)\leq \la \leq -f'(0)+2\ep.
$$

Now we claim that $\la >0$ for large enough $\alpha$. Argue by contradiction and assume that $\la \leq 0$ for all $\alpha \in
(0,+\infty)$.  Since
$$
0\geq \la=\int |\nabla \pa |^2+\alpha \int g \pa^2 - f'(0),
$$ we get
$$
 \int g \pa^2 \leq \frac{1}{\alpha} f'(0)
$$
and $\pa \rightarrow 0$ in $L^2(\R^{N-1} \setminus B_R)$ for all $R>0$. Furthermore, $\pa$ is bounded in $\HH$ and up to extraction we can assume that $\pa$ converges strongly in $L^2(\R^{N-1})$, thus $\pa$ converges to $0$ in $L^2$ but this is impossible since $\int {\pa}^2=1$ for all $\alpha>0$.
\end{proof}
\begin{corollary}
\label{alpha0}
There exists $\alpha_0 >0$ such that $\la <0$ for $\alpha <\alpha_0$, $\lambda_{\alpha_0}=0$
and $\la > 0$ for $\alpha >\alpha_0$. 
\end{corollary}

\subsection{If $g$ vanishes on $B_{r_0}$}
\label{gvanish}
The main part of the proof still holds if $g$ vanishes on $B_{r_0}$ but the result is slightly modified.

In this section, we assume that there exists $r_0>0$ such that \eqref{Hypgpos} is substituted by the following assumption
\begin{equation}
\label{Hypg'}
 g:\R^{N-1} \rightarrow \R_+ \in \CC^0, \;  g\equiv 0 \text{ on } B_{r_0} \text{ and } g>0 \text{ on }\R^{N-1}\setminus B_{r_0}. 
\end{equation}
We define $\lambda_\Delta$ the principal eigenvalue of the Laplacian on $B_{r_0}$ with Dirichlet boundary conditions, i.e.
$$
\begin{cases} 
-\Delta \phi_0=\lambda_\Delta \phi_0 & \text{on }B_{r_0},\\
\phi_0=0 & \text{on } \partial B_{r_0}.
\end{cases}
$$

In this case, the principal eigenvalue of the linearized operator about 0 is well defined and Proposition \ref{lacomp} becomes
\begin{proposition}
The function $\alpha \mapsto \la$ is continuous, increasing and concave
for $\alpha \in (0, +\infty)$, and $\lim_{\alpha \rightarrow 0}
      \la =-f'(0)$. Now there are two cases:\\
i) If $f'(0) < \lambda_\Delta$, then for $\alpha$ large enough $\la>0$. \\
ii) If $f'(0)\geq \lambda_\Delta$, then $\la\leq 0$ for all $\alpha>0$.
\end{proposition} 
\begin{proof}
The proof of the first part of the proposition is exactly the same as in Proposition \ref{lacomp}. We just have to prove i) and ii).

i) We assume that $f'(0) < \lambda_\Delta$ and argue by contradiction assuming that $\la \leq 0$ for all $\alpha \in (0, +\infty)$. As in the proof of proposition \ref{lacomp}, we have 
$$
\int g \pa^2 \leq \frac{1}{\alpha} f'(0)
$$
and this yields $\pa \rightarrow 0$ in $L^2(\R^{N-1}\setminus B_R)$ for $\alpha \rightarrow +\infty$ but now for all $R>r_0$ only.

As before $\pa$ is bounded in $\HH$ and up to extraction, we have $\la \rightarrow \lambda\leq 0$, weak convergence in $\HH$ and strong convergence in $L^2$ of $\pa$ to $\phi$. The limit $\phi$ verifies $\int \phi^2=1$, $\phi \equiv 0$ for $|y|>r_0$ and 
$$
-\Delta \phi-f'(0) \phi=\lambda \phi
$$ 
Thus $\phi \in H^1_0(B_{r_0})$ must coincide with $\phi_0$ in $B_{r_0}$ and $\lambda +f'(0)=\lambda_\Delta$ leading to $f'(0)\geq \lambda_\Delta$ since $\lambda \leq 0$. This is a contradiction.

ii) By taking $\varphi=\phi_0$ in the Rayleigh quotient \eqref{QR}, where $\phi_0$ is the principal eigenvalue of the above problem in $B_{r_0}$ with Dirichlet boundary conditions, we see that $\lambda_\alpha\leq \lambda_\Delta -f'(0) \leq 0$ for all $\alpha>0$. 
\end{proof}

In the following, we will not state the results specifically for this case \eqref{Hypg'} and will rather assume \eqref{Hypgpos}. However, the proofs and results developed here carry over to this case with the obvious modifications.

\subsection{Existence of non-zero asymptotic profile}


\begin{theorem}
\label{existPA}
For $\alpha \geq \alpha_0$, there is no solution of \eqref{PA}, where $\alpha_0$ is defined in corollary \ref{alpha0}. 
For $\alpha < \alpha_0$, there exists a unique positive
solution of \eqref{PA}. 
\end{theorem}
\begin{proof}
Let us fix $\alpha \geq \alpha_0$. Then $\la \geq  0$.
Assume by contradiction that there exists a solution $V$
of \eqref{PA}. Then the strong maximum principle shows that $V>0$. 

Since $\pa$ is an eigenfunction of the linearized operator $L$ and
$V$ is solution of \eqref{PA}, we have 
\begin{eqnarray*}
\int \left(\Delta V +f(V)-\alpha gV \right)\pa & = & 0 \\
  & =  & \int  (\Delta \pa +(f'(0)-\alpha g)\pa + \la \pa)V
\end{eqnarray*}
Now from corollary \ref{decry}, $V$ and
$\nabla V$ are rapidly decreasing for
$|y|\rightarrow \infty$ and so we can apply Stokes formula
$\int \Delta V \, \pa= \int V \,\Delta \pa$. It yields
$\displaystyle \int (f(V)-f'(0)V) \pa = \la \int \pa V$ but
$f(V)-f'(0)V<0$ since $f$ is of Fisher-KPP type and $ \la \geq 0$ thus a
contradiction is obtained.

\bigskip
We now turn to the case $\alpha<\alpha_0$.
For $\alpha<\alpha_0$, the eigenvalue $\la$ is negative. Setting
$\psm=\ep \pa$ with $\ep>0$,
we get 
$$
\Delta \psm+f(\psm) - \alpha
g \psm=-\la \ep \pa+f(\ep\pa)-f'(0)\ep\pa \geq 0
$$ 
if $\ep>0$ is chosen small enough.
Hence $\psm$ is a sub-solution of \eqref{PA}. The constant function 1 is
a super-solution and 
$\psm\leq 1$ if $\ep$ is small enough. Therefore by the sub- and super-solution method,
there exists a 
solution $V$ such that $0<\psm\leq V \leq 1$.

Now consider $V$ and $W$ two non-zero solutions of
\eqref{PA}. We argue by contradiction and assume that $V \not\equiv W$. Then
 for example  $\Omega =\{ y\in \R^{N-1}, \,
V(y)<W(y)\}$ is not empty. Introduce a cutoff function
$\beta \in \CC^\infty (\R)$ with $\beta=0$ on $(-\infty, 1/2]$,
$\beta=1$ on $[1,+\infty)$ and $0<\beta'<4$ on $(1/2,1)$ and for all $\ep
>0$, let us set $\beta_\ep(s)=\beta \left( \frac{s}{\ep} \right)$.  
Using equation \eqref{PA}, we have
\begin{eqnarray*}
\int (-V \Delta W+ \Delta V W)\beta_\ep(W - V) 
&=& \int (V f(W)-f(V)W)\beta_\ep(W - V) \\
& &\xrightarrow{\ep \rightarrow 0} \int_{\Omega}(V f(W)-f(V)W) 
\end{eqnarray*}
by Lebesgue's dominated convergence theorem.
Owing to corollary \ref{decry}, $V$, $\nabla V$, $W$ and $\nabla W$ have exponential decay and thus Stokes
formula can be applied and we obtain
\begin{eqnarray*}
\int (-V \Delta W+ \Delta V W)\beta_\ep(W - V) 
= \int \beta_\ep '(W -V) \nabla
(W-V) .\left( V \nabla W - W \nabla V \right)\qquad \\
= \underbrace{\int \beta_\ep ' (W -V) V |\nabla (W
-V)|^2}_{=I_1} - \underbrace{\int \beta_\ep ' (W -V)(W -V) \nabla (W
-V). \nabla V}_{=I_2} .
\end{eqnarray*}
In the term $I_2$ the integrand satisfies 
$$
|\beta_\ep ' (W -V)(W-V)\nabla(W-V).\nabla W| \leq 4|\nabla(W-V)|. |\nabla W|
$$
Therefore by Lebesgue's Theorem of dominated convergence, we infer that $I_2\rightarrow 0$. Next the term $I_1$ satisfies $I_1\geq 0$. Consequently, we may write: 
$$
0\geq \int_{\Omega}\big( V f(W)-W f(V) \big)=\int_{\Omega}\left( \frac{f(W)}{W}-\frac{f(V)}{V}\right) VW
$$ 
which is a contradiction in view of \eqref{HypfFisher-KPP} as $W>V$ in $\Omega$. Hence $V=W$ and
the non-zero solution is unique.
\end{proof}



The last point concerns the stability of the asymptotic profiles
for $\alpha <\alpha_0$. Let us start by studying the energy of
$V$. For $w\in \HH$, we define the energy 
\begin{equation}
J_\alpha (w)=\int_{\R^{N-1}} |\nabla w(y)|^2 +\frac{\alpha}{2} g(y) w^2(y) -
F(w(y)) \,dy
\label{defenergie}
\end{equation} 
where $F(u)=\int_0^u f(t)dt$.
\begin{theorem}
For $\alpha<\alpha_0$, the unique positive solution
of \eqref{PA} $V$ is stable in the energy sense, i.e. $V$ is the
global minimum of $J_\alpha$ and, furthermore $J_\alpha (V) <0=J_\alpha(0)$.
\label{energy}
\end{theorem}
\begin{proof}
Owing to the maximum principle, solutions of \eqref{PA} are between 0 and
1. Hence we can modify $f$ on $]-\infty,0[$ such that it becomes odd
and as a consequence, $F$ can be considered as even.
Since $\la$ the principal eigenvalue of the linearized operator
about the zero solution is negative for $\alpha<\alpha_0$, 0 cannot be the global minimum of $J_\alpha$. Now $J_\alpha$ admits a global minimum that
will be called $\tV$ for the argument. One can prove that $|\tV|$ is
also a global minimum of $J_\alpha$ and hence $|\tV|$ is a positive
solution of \eqref{PA}. By uniqueness, $|\tV|=V$ and thus
$V$ is a global minimum of $J_\alpha$. Since $0$ is not a
global minimum, necessarily $J_\alpha (V) <0=J_\alpha(0)$.
\end{proof} 
We now conclude with the linearized stability of $V$.
\begin{theorem}
For $\alpha <\alpha_0$, consider the linearized operator about
$V$ and denote $\lambda_1[V]$ the principal
eigenvalue of this operator. Then $\lambda_1[V]> 0$. 
\end{theorem}
\begin{proof}
Denote by $\psi$ a positive eigenfunction associated with
$\lambda_1[V]$ and assume by contradiction that
$\lambda_1[V]\leq 0$. If $\lambda_1[V]<0$, it is easy to see that for
$\ep>0$ small enough $V +\ep \psi<1$ is a
sub-solution of \eqref{PA}. From there, it would follow that there exists a solution
of \eqref{PA} between $V +\ep \psi$ and $1$ but this
contradicts the uniqueness of $V$.

Now if $\lambda_1[V]=0$,letting $\psi$ be as above, we get 
\begin{equation}
\label{psi}
-\Delta \psi +\alpha g(y) \psi-f'(V)\psi=0.
\end{equation}  
From the equation and since $V$ is unique for every  given $0<\alpha<\alpha_0$, it is clear that $V$ is differentiable with respect to $\alpha$ and that $w:= \frac{\partial V}{\partial \alpha}$ satisfies:
\begin{equation}
\label{ww}
-\Delta w +\alpha g(y) w-f'(V)w=-g(y)V.
\end{equation}  
We know that $w\leq 0$ and from \eqref{ww} which shows that $w\not\equiv 0$, we actually see from the maximum principle that $w<0$ in $\R^{N-1}$.
It is also easily seen that $w$ has exponential decay at infinity. 
From \eqref{psi} and \eqref{ww}, it then follows that $\int_{R^{N-1}} gVw=0$ which is a contradiction. Hence $\lambda_1[V]>0$. 

\end{proof}


\section{Traveling fronts for a Fisher-KPP non-linearity}
\label{section: cetoile}

This section is devoted to the definition of a speed $\cu$ for which a
traveling front of equation \eqref{EP} exists for $\alpha\in (0,\alpha_0)$. The threshold
of existence of the non-zero asymptotic profile is called
$\alpha_0$ as in the previous section. For $0<\alpha <\alpha_0$,
$V$ denotes the unique non-zero asymptotic
profile. As shown in the previous section, the
energy of the non-zero profile $J_\alpha(V)$ is
negative.

A curved traveling front of speed $c$ is a function
$u(x_1-ct,y)$ solution of 
equation \eqref{EP} and connecting the non-zero asymptotic state $V$
to 0. Thus we are looking for a solution of 
\begin{equation}
\begin{cases}
-\Delta u - c\partial_1 u + \alpha g(y) u=f(u), \quad  x=(x_1,y)\in
\Rn \\
u(x_1,.)\xrightarrow{x_1\rightarrow -\infty} V, \quad
u(x_1,.)\xrightarrow{x_1\rightarrow +\infty} 0 \; \text{uniformly in }
y\in \R^{N-1},\\
u\geq 0, \quad u \text{ bounded}
\end{cases}
\label{FPc}
\end{equation} 
where $c\in \R$ is also an unknown of the problem.

The construction of $\cu$ in Theorem \ref{existFP} uses the sliding method following ideas of  \cite{BN}. Note however that there are important differences with \cite{BN}. In that paper, the Fisher-KPP case is derived by first solving the ``combustion non-linearity'' and then approach the Fisher-KPP non-linearity as a limiting case of truncated functions. Contrary to \cite{BN} here, we derive directly the existence of a solution of the Fisher-KPP case. Actually the method we present here can be applied to somewhat simplify the proof of \cite{BN} in the Fisher-KPP case for cylinder with Neumann conditions. 
  
\subsection{Problem on a domain bounded in $x_1$.}
Let us fix $a>1$ and $c\in \R$ for this subsection and consider the
following problem: 
\begin{equation}
\begin{cases}
-\Delta u - c\partial_1 u + \alpha g(y) u=f(u), \quad  x=(x_1,y)\in
(-a,a)\times \R^{N-1} \\
u(-a,\cdot)=V, \quad u(a,\cdot)= 0, \\
 u \geq 0,  \quad u \text{ bounded}.
\end{cases}
\label{FPB}
\end{equation} 
The aim of this subsection is to prove the following theorem:
\begin{theorem}
\label{boite}
There exists a unique solution of (\ref{FPB}), denoted $u_a^c$ in
the following. This solution decreases in the $x_1$-direction,
i.e. $\partial_1 u_a^c<0$. Thus $0 <u_a^c<V$ on $(-a,a)\times
\R^{N-1}$. Moreover $c \mapsto u_{a}^c$ is decreasing and 
continuous from $\R$ to $L^\infty([-a,a]\times
\R^{N-1})$. 
\end{theorem}

To prove this theorem, we require the following two
propositions.
\begin{proposition}
\label{majoration}
Let $u$ be a solution of \eqref{FPB}. Then $u(x_1,y) \leq V(y)$ for
$(x_1,y)\in  [-a,a]\times \R^{N-1}$.
\end{proposition}
\begin{proof}
Let $M\geq 1$ be such that $u\leq M$ and consider $\psi_R$ defined on
$B_R$ the largest solution of
\begin{equation}
\label{psiR}
\begin{cases}
-\Delta_y \psi_R +\alpha g(y)\psi_R=f(\psi_R) \quad \text{for } y\in B_R,\\
\psi_R=M \quad \text{for }y \in \partial B_R, \qquad 0 \leq \psi_R
\leq M.
\end{cases}
\end{equation}
Here we think of $f$ as having been extended by $0$ outside $[0,1]$. Since $f(s)\leq 0$ for all $s\geq 1$, we observe that:
\begin{itemize}
\item by the strong maximum principle, $0<\psi_R<M$ on $B_R$.
\item since $V \leq 1 \leq M$ and $V$ is a sub-solution of
  \eqref{psiR}, through monotone iterations we have $V\leq  \psi_R$. 
\item if $R'>R$, $\psi_{R'}$ is once again a sub-solution of
  \eqref{psiR} on $B_R$ and thus $\psi_{R'}\leq \psi_R$ on $B_R$.
\item therefore $\psi_R$ tends to a function when
  $R\rightarrow +\infty$ and through local elliptic estimates, this
  function is a non-zero solution ($\geq V$) of the asymptotic problem
  \eqref{PA}. By uniqueness, we obtain $\psi_R
  \xrightarrow{R \rightarrow +\infty} V$   
\end{itemize}
Now we consider the problem
\begin{equation}
\label{recw}
\begin{cases}
-\Delta w -c\partial_1 w+\alpha
g(y)w=f(w) \quad \text{ for } x \in (-a,a)\times B_R,\\
w=M \quad \text{for }x\in (-a,a)\times \partial B_R, \qquad w=\psi_R \quad
\text{for } x_1=\pm a, y\in B_R.
\end{cases}
\end{equation}
The solution $u$ of \eqref{FPB} is a sub-solution of \eqref{recw} and
the constant function $M$ is a super-solution. Using monotone iterations
starting from the super-solution $M$, we build the same sequence as
previously (for problem \eqref{psiR}) since by induction the solutions
do not depend on $x_1\in (-a,a)$. Hence the sequence converges toward
$\psi_R$ and we have $u\leq \psi_R \leq M$. Now letting $R \rightarrow +\infty$
yields $u\leq V$.
\end{proof}
\begin{proposition}
\label{comparison}
Let $R$ be such that $g(y) >\frac{K}{\alpha}$ for $y \not\in \overline{B_R}$ where $K$ is the
Lipschitz norm of $f$ on $[0,1]$. We set $\Omega=I\times
(\R^{N-1}\setminus \overline{B_R})$ where $I$ is an open bounded interval
of $\R$.  

Suppose $u$ and $v\in \CC^2(\Omega)\cap \CC^0(\overline{\Omega})$ are
solutions of 
\begin{equation}
\label{jeq}
-\Delta w - c\partial_1 w + \alpha g(y) w=f(w) \quad \text{on } \Omega
\end{equation}  
and $u\leq v$ on $\partial \Omega$. Then $u\leq v$ on $\Omega$.  
\end{proposition}
\begin{proof}
By contradiction, suppose this is not true. Due to corollary
\ref{decry} and proposition \ref{majoration}, $u(x_1,y)$ and
$v(x_1,y)$ converge uniformly to 0 for $|y|\rightarrow +\infty$. Consequently, there
exist $(x_0,y_0)\in \Omega$ such that
$$
0>\min_{\overline{\Omega}} (v-u) =(v-u)(x_0,y_0). 
$$
Since $(x_0,y_0)\in \Omega$, we have $\partial_1 (v-u)(x_0,y_0)=0$
and $\Delta (v-u)(x_0,y_0)\geq 0$, and subtracting the equation \eqref{jeq} with $u$
from the one with $v$, we obtain  
$$
\alpha g(|y_0|)(v-u)(x_0,y_0) \geq f(v(x_0,y_0))-f(u(x_0,y_0)) \geq -K
|(v-u)(x_0,y_0)| 
$$
which is impossible since $\alpha g(|y_0|) >K$ and $(v-u)(x_0,y_0)<0$.   
\end{proof}

Let us now turn to the proof of Theorem \ref{boite} using
sliding method.
\medskip\\
First $\us(x,y)=V(y)$ is a super-solution, 0 is a
sub-solution and $0\leq \us$, so by monotone iterations, there
exists a solution $u$ of \eqref{FPB}. 

\begin{lemma}
Assume $u$ and $v$ are two
solutions of (\ref{FPB}). Then
$$
v(x_1+h,y)\leq u(x_1,y) \text{ for all } h \in [0,2a) \text{ and all } (x_1,y)\in [-a,a-h]\times \R^{N-1}.
$$
\end{lemma}
\begin{proof}[Proof of the lemma]
By proposition \ref{majoration}, we have
$0\leq u\leq V$ (resp. $0\leq v \leq V$) and using the strong
maximum principle, we obtain $0<u<V$ (resp. $0<v<V$) on
$(-a,a)\times \R^{N-1}$. 
 
For $h\in [0,2a)$, let $I_h=(-a,a-h)$ and for $(x_1,y)\in \overline{I_h}\times \R^{N-1}$, set $v_h(x_1,y)=v(x_1+h,y)$. 

Let us fix $R>0$ such that $g(y
)>\frac{K}{\alpha}$ for $y \not\in \overline{B_R}$. By compactness and
continuity of $u$ and $v$, there exists $\ep>0$ such that $v_h \leq
u$ on $\overline{I_h} \times \overline{B_R}$ for any $h$ such that $2a-\ep\leq h < 2a$. Proposition \ref{comparison} shows that $v_h \leq u$ on $\overline{I_h} \times
\R^{N-1}$ for any $h\geq 2a-\ep$.
This enables us to define $$h^*= \inf \{h\geq 0, \, v_h\leq u \text{ on }
\overline{I_h} \times \R^{N-1}\}.$$ 

Let us prove that $h^*=0$ and argue by contradiction that $h^*>0$. 
By continuity, $v_{h^*}\leq u$ on $\overline{I_{h^*}} \times \R^{N-1}$. 

Suppose that $\displaystyle \min_{I_{h^*} \times B_R} u-v_{h^*}>0$. This would imply that for $h^*-h>0$ small, $\displaystyle \min_{I_{h} \times B_R} u-v_{h}>0$ and by Proposition \ref{comparison}, $v_h\leq u$ on $I_h \times \R^{N-1}$ in contradiction with the definition of $h^*$. Therefore $\displaystyle \min_{I_{h^*} \times B_R} u-v_{h^*}=0$. This implies the existence of $(x_1^*,y^*)\in I_{h^*}\times \overline{B_R}$ such that $v_{h^*}(x_1^*,y^*)=u(x_1^*,y^*)$ (note that $u-v_{h^*}>0$ for $x_1=-a$ or $a-h^*$). Writing in the usual way that $u-v_{h^*}$ is solution of a linear elliptic equation in $I_{h^*}\times \R^{N-1}$ and $u-v_{h^*}\geq 0$ with $u-v_{h^*}$ vanishing at the point $(x_1^*,y^*)$, the strong maximum  principle implies that $u-v_{h^*}\equiv 0$ which is impossible.
\end{proof}

Applying the preceding lemma with $h=0$ yields the uniqueness of
the solution of \eqref{FPB}. Taking $u=v=u_a^c$, one sees that $u_a^c$ is monotone decreasing. Thus $\partial _1 u_{a}^c \leq 0$ and deriving
equation \eqref{FPB} and applying once more the maximum principle gives
$\partial _1 u_{a}^c<0$.  

It remains to study the behavior of $u_{a}^c$ with respect to
$c$. The continuity is deduced from the uniqueness of the
solution and a priori estimates in the standard way. Now let $c_1<c_2$ and denote by $u_1$ (resp. by $u_2$) the
solution of \eqref{FPB} with $c=c_1$ (resp. $c=c_2$). Since $\partial _1 u_1<0$,
$$
\Delta u_1 +c_2 \partial_1 u_1 +f(u_1)-\alpha g(y)
u_1=(c_2-c_1)\partial_1 u_1 <0
$$
and $u_1>0$ is a super-solution of equation \eqref{FPB} with
$c=c_2$. By uniqueness of the solution, necessarily $u_2\leq
u_1$. Once more the strong maximum principle implies $u_2<u_1$. 

\subsection{Convergence to a solution on $\Rn$}

Now that the equation is solved on a domain bounded in the $x_1$-direction, the idea is
to increase $a$ up to infinity so that the domain
tends to $\Rn$. However if $c$ is chosen 
arbitrarily, the function $u_a^c$ may converge toward the constant 0
or to $V$ when $a$ tends to infinity. Hence we adopt a normalization method as in \cite{BN}.  The following
theorem will define the value of the speed $c$ depending on $a$ to
avoid those situations. We recall that since $\alpha<\alpha_0$,
$\lambda_\alpha$ the principal eigenvalue of the linearized operator
about the solution 0 is negative. 

\hspace{-7mm}
\begin{minipage}{0.62\textwidth}
\begin{theorem}
\label{ca}
Let us fix $\ep>0$. Let $\delta >0$ be such that $\delta
< -\lambda_\alpha\leq f'(0)$. Let $\eta >0$ be such that $\forall
s\in [0, \eta] \; f(s)\geq (f'(0)-\delta)s$. We fix $\theta \in
(0, \frac{\eta}{2})$.\\
Then there exists $A(\ep)>0$ such that for all $a\geq A(\ep)$, there exists a
unique speed $c_a\in (0,2\sqrt{-\la}+\ep)$ with
$u_a^{c_a}(0,0)=\theta$.     
\end{theorem}
\end{minipage} \hfill 
\begin{minipage}{0.35\textwidth}
\includegraphics[width=5cm]{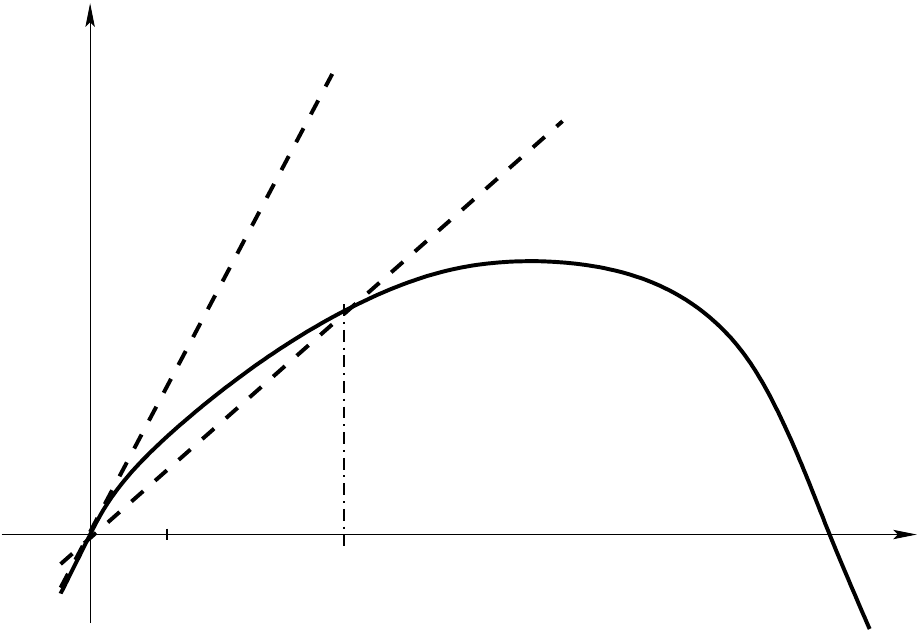}
\end{minipage}
\begin{proof}
By continuity and monotonicity, it suffices to prove: 
\begin{enumerate}[i)]
\item $u_a^0(0,0)>\theta$, 
\item $u_a^c(0,0)<\theta$ for $c=2\sqrt{-\la}+\ep$
and $a$ large enough.
\end{enumerate}

i) Assume $c=0$. Let $\pa$ be the positive eigenfunction
of the linearized operator $L$ associated with the first eigenvalue
$\la <0$ and with the normalization $\|\pa\|_\infty=1$. 
Let us introduce
$v(x_1,y)=h(x_1)\pa(y)$ for $(x_1,y)\in [-a,a]\times \Rn$ where
$h(x_1)=\eta \frac{a-x_1}{2a}$. Then $0<v\leq \eta$ on $[-a,a]\times \Rn$,
which yields
\begin{eqnarray*}
-\Delta v + \alpha g(y) v-f(v) &\leq& -\Delta v+\alpha g(y)v -
 (f'(0)-\delta)v \\
& = & (\la+\delta)v\leq 0.
\end{eqnarray*}
Moreover by construction of $V$ (cf section \ref{section: asympt}), $v(-a,y)=\eta \pa (y)<V(y)$ if $\eta$ is small enough. Then
$v(a,y)=0$ and $v\leq 1$. Hence, $v$ is a sub-solution of \eqref{FPB} for $c=0$. Thus
$v\leq u_a^0$ and therefore, $u_a^0(0,0) \geq
v(0,0)=\frac{\eta}{2}>\theta$.

ii) Let us construct an explicit super-solution for 
$c=2\sqrt{-\la}+\ep$. We recall from section \ref{principal} that
$\lambda_\beta <\la$ for $\beta<\alpha$ and that
$\lim_{\beta\rightarrow \alpha}\lambda_\beta=\la$.  
Thus there exists $\beta\in
(0,\alpha)$ such that $2\sqrt{-\lambda_\beta}\leq 2\sqrt{-\la}+\ep$.
As before, let $\psi_\beta$ denote the positive eigenfunction
of the linearized operator $L$ associated with the first eigenvalue
$\lambda_\beta <0$ with the normalization $\psi_\beta(0)=1$. Choose $R$ such that for all $r\geq R$,  $(\alpha
-\beta)g(r)+\lambda_\beta>0$ and $\alpha g(r)>f'(0)$, and choose $k>0$ such
that $k \psi_\beta \geq V$ on $\overline{B}_R$. The constant $k$ only depends
on $\beta$ hence on $\ep$.

\begin{lemma}
Then $k \psi_\beta\geq V$ on $\R^{N-1}$.
\end{lemma}
\begin{proof}[Proof of the lemma]
 We follow a similar proof to that of lemma \ref{comparison}:  \\
 If the lemma does not stand, since
$k\psi_\beta-V$ tends to 0 at $\infty$, there exists $y_0 \in\R^{N-1}\setminus \overline{B}_R$
such that $(k\psi_\beta-V)(y_0)=\min_{\R^{N-1}} (k\psi_\beta-V)<0$. At
this point, $\Delta (k\psi_\beta-V)\geq 0$
 but 
\begin{eqnarray*}
-\Delta (k\psi_\beta-V)+\alpha g(y_0)
(k\psi_\beta-V)-f'(0)(k\psi_\beta-V)
= \qquad \qquad \qquad \qquad  \\
\big((\alpha-\beta)g(|y_0|) +\lambda_\beta\big) k\psi_\beta
+f'(0)V-f(V) 
\geq \big((\alpha-\beta)g(|y_0|)+\lambda_\beta\big)k\psi_\beta.
\end{eqnarray*}
By the choice of $R$, we get $\Delta (k\psi_\beta-V)(y_0)<0$ which yields
a contradiction. 
\end{proof}

Let us now build the super-solution when $c=2\sqrt{-\la}+\ep$. We set
$w(x_1,y)=z(x_1)k\psi_\beta(y)$ where $z$ is the solution of 
$$
\begin{cases}
z''+cz'-\lambda_\beta z=0 \quad \text{on } (-a,a), \\
z(-a)=1, \quad z(a)=0.
\end{cases}
$$
Then $w$ verifies
$$
\begin{cases}
-\Delta w-c\partial_{1}w+\alpha
g(y)w=(\alpha-\beta)g(y)w+f'(0)w\geq f(w) \\
w(-a,.)=k\psi_\beta\geq V, \quad w(a,.)=0.
\end{cases}
$$
so it is indeed a super-solution of \eqref{FPB}. Moreover
$$
z(x)=\frac{e^{\rho_+(x-a)}-e^{\rho_-(x-a)}}{e^{-\rho_+2a}-e^{-\rho_-2a}} \geq 0
$$
where $\rho_-<\rho_+<0$ are the roots of $\rho^2+c\rho-\lambda_\beta=0$,
i.e. $\rho_\pm = \frac{-c\pm\sqrt{c^2+4\lambda_\beta}}{2}$ (note that $c^2+4\lambda_\beta\geq 0$). Hence 
\begin{eqnarray*}
0 <
u_a^c(0,0)&<&w(0,0)=\frac{e^{-\rho_+a}-e^{-\rho_-a}}{e^{-\rho_+2a}-e^{-\rho_-2a}}k\psi_\beta(0)
\\
 &=&\frac{1}{e^{-\rho_+a}+e^{-\rho_-a}}k\psi_\beta(0)\leq e^{-\frac{c}{2}a}k \leq e^{-a\sqrt{-\lambda_\alpha}}k
\end{eqnarray*}
Thus if $a$ is large enough to have
$e^{-a\sqrt{-\lambda_\alpha}}<\frac{\theta}{k}$, then for $c=2\sqrt{-\la}+\ep$, we get  $u_a^c(0,0)<\theta$.
\end{proof}
With the bounds on the speed $c_a$ it is now possible to pass to the limit as  $a$ tends to infinity.
\begin{proposition}
\label{ainfty}
There exists a sequence $(a_j)_{j\in \N}$ such that
$a_j \rightarrow +\infty$, $c_{a_j}\rightarrow \cu\in [0, 2\sqrt{-\la}]$  and
$u_{a_j}^{c_{a_j}}\rightarrow u$ in $\CC^2_{loc}(\Rn)$. The limit $u$
is solution of
\begin{equation}
\label{systainf}
\begin{cases}
-\Delta u -\cu \partial_1 u +\alpha g(y) u = f(u) \quad \text{on } \Rn \\
0 \leq u\leq V, \quad u(0,0)=\theta, \quad \partial_1 u \leq 0. 
\end{cases}
\end{equation}
Then $u$ is necessarily a traveling front solution of \eqref{FPc} with
$c=\cu$. 
\end{proposition}
\begin{proof}
First for $j\in \N$, fix $\ep_j=\frac{1}{j+1}$ and $a_j\geq
A(\ep_j)$ with $a_j\rightarrow +\infty$.
Since $c_{a_j}\in (0,2\sqrt{-\la}+\ep_j]$ is bounded,
$u_{a_j}^{c_{a_j}}$ is uniformly bounded in 
$\CC^{2,\gamma}$ for any $\gamma \in (0,1)$. Hence up to an extraction of a subsequence,
there exist $\cu\in [0,2\sqrt{-\la}]$ and $u\in \CC^2_{loc}$ such that
$c_a\rightarrow \cu$ and $u_a^{c_a}\rightarrow u$. Clearly the function $u$ is
a solution of \eqref{systainf}. Owing to the normalization $u(0,0)=\theta$, $u$ is not a constant, moreover by the maximum
principle $0<u<V$ and $\partial_xu<0$. Since $u$ is decreasing in
$x$, $\displaystyle u_\pm =\lim_{x\rightarrow \pm \infty} u(x,\cdot)$ are
solutions of \eqref{PA} and $u_-(0)>\theta>0$ and $0 \leq u_+
(0)<\theta$. This implies that $u_-=V$ and $u_+\equiv
0$. Thus $u$ is indeed a traveling front solution of \eqref{FPc}.
\end{proof}

\subsection{Existence of traveling front for $c\geq 2\sqrt{-\la}$.}
\label{section: exist}

In this section we still assume $0<\alpha <\alpha_0$ and we will
prove the following theorem. 
\begin{theorem}
\label{existc*}
There exists a
traveling front of speed $c$ of equation \eqref{FPc} if and only if $c\geq
2\sqrt{-\la}$. 
\end{theorem}
We start with the Proposition
\begin{proposition}
\label{minc*}
For $c<2\sqrt{-\la}$ there exists no traveling front solution of
\eqref{FPc}. Thus $\cu=2\sqrt{-\la}$ (where $\cu$ is the traveling speed constructed in the previous section).
\end{proposition}
\begin{proof}
We argue by contradiction and assume that there exists a traveling front $u$ of
speed $c<2\sqrt{-\la}$ of \eqref{FPc}. 
We are going to construct a
small positive sub-solution with compact support.
To this end, we can find $\delta\in (0, f'(0))$ such that
$c^2+4(\la+2\delta)<0$  and
$\eta>0$ such that for all $s\in [0,\eta]$, $f(s)\geq (f'(0)-\delta) s$.

Since the linearized operator $L=-\Delta +\alpha g(y)-f'(0)$ is
self adjoint, the principal eigenvalue $\la$ is the limit of the
Dirichlet principal eigenvalue in $B_R$ when $R\rightarrow \infty$ (see \cite{HBLR06} for more details):
\begin{equation}
\label{vpR}
\begin{cases}
-\Delta \psi^R+\alpha g(y) \psi^R-f'(0)\psi^R=\la^R \psi^R, \quad y\in
B_R, \\
\psi^R>0 \quad \text{on } B_R, \quad \psi^R=0 \quad \text{on } \partial B_R.
\end{cases}
\end{equation}
Precisely $\la^R >\la$ and $\la^R \xrightarrow{R\rightarrow \infty}
\la$. In the following, $\psi^R$ denotes the positive eigenfunction with
$\|\psi^R\|_\infty=1$ and let us fix $R$ sufficiently large so that
$\la <\la^R<\la+\delta$.  

Let $\tilde{\sigma}=\sigma+i \frac{\pi}{2L}$, $L>0$, be an imaginary root of 
$
X^2+c X -\la^R -\delta =0
$
which is possible since $c^2 +4(\la^R+\delta) <c^2 +4(\la +
2\delta)<0$. Finally let us fix $\ep>0$
small enough such that $\ep e^{\sigma x_1}<\eta$ and $\ep e^{\sigma
  x_1} \psi^R(y)<u(x_1,y)$ for $x\in [-L,L]$ and $y\in \overline{B}_R$.
We set 
\begin{equation}
\label{constrss}
w(x_1,y)= \begin{cases}
\ep e^{\sigma x_1} \cos(\frac{\pi}{2L}x_1) \psi^R(y)  \text{ if }
-L<x_1<L, \; y\in B_R, \\
0  \text{ otherwise.}
\end{cases}
\end{equation}
Then $w$ verifies
$$
-\Delta w - c\partial_1 w+\alpha g(y) w=(f'(0)-\delta)w \leq f(w)
$$
since $0 \leq w \leq \eta$. Moreover $w\leq
u$ and $w>0$ on $(-L,L)\times B_R$. Thus $w$ is a generalized sub-solution
with compact support \cite{BL80}.

Let us now derive a contradiction with the existence of a traveling
front $u$. Translate $u$ to the left by defining
$u_\tau(x_1,y)=u(x_1+\tau,y)$ for $\tau>0$.  Since $u(x_1,.)
\xrightarrow{x_1\rightarrow +\infty} 0$, there exists $\tau^*\geq
0$ such that $u_{\tau^*}\geq w$ but
$u_{\tau^*}(x_1^*,y^*)=w(x_1^*,y^*)$. Since $u_{\tau^*}>0$, $x_1^*\in (-L,L)$
and $y^*\in B_R$ (an interior point of the support of $w$). Now since
$w$ is a sub-solution, the strong maximum principle yields
$u_{\tau^*}\equiv w$ on $[-L,L]\times \overline{B_R}$, but this is
impossible on the boundary. 
\end{proof}

We have already proved that for $c<2\sqrt{-\la}$, there exists no traveling front of speed
$c$ solution of \eqref{FPc} and that for $c=\cu=2\sqrt{-\la}$ there
exists a traveling front of speed $c$. Let us prove that for any $c>\cu$
there exists at least a traveling front to conclude with Theorem \ref{existc*}. The proof goes as usual. We consider
the following problem
\begin{equation}
\begin{cases}
-\Delta u - c\partial_1 u + \alpha g(y) u=f(u), \quad  x=(x_1,y)\in
(-a,a)\times \R^{N-1} \\
u(-a,\cdot)=u^*(-a+r,\cdot), \quad u(a,\cdot)= u^*(a+r,\cdot)
\end{cases}
\label{FPar}
\end{equation}
where $u^*$ is the traveling front of speed $c^*$.
The function $u^*(\cdot+r,\cdot)$ is a strict super-solution of \eqref{FPar}
(since $c>c^*$) when 0 is a strict sub-solution and $0<u^*(\cdot+r,\cdot)$. Hence as in
theorem \ref{boite}, it can be proved that there exists a unique
solution $v_a^r$ of \eqref{FPar} and moreover $\partial_x w_a^r<0$ and
$$\forall (x_1,y)\in [-a,a]\times \R^{N-1} \quad V(y)>u^*(-a+r,y)\geq
v_a^r(x_1,y)\geq u^*(a+r,y)>0.$$ 
By uniqueness, $w_a^r$ depends continuously on $r\in \R$, so
$w_a^r \xrightarrow{r\rightarrow +\infty}0$ and
$w_a^r \xrightarrow{r\rightarrow -\infty} V$ uniformly on
$[-a,a]\times \R^{N-1}$. Let us denote $u_a=v_a^r$ where $r$ is chosen
in order that $v_a^r(0,0)=\theta$ (see previous section for definition
of $\theta$). Once again taking any sequence $a_n\rightarrow +\infty$,
up to an extraction $u_{a_n} \rightarrow u$ in $\CC^2_{loc}$ and $u$ is
a traveling front of speed $c$ solution of \eqref{FPc}.


\section{The case of a  Fisher-KPP non-linearity. Asymptotic speed of spreading.}
\label{section: convergence}

This section is concerned with the asymptotic behavior of the
solutions of the parabolic problem
\begin{equation}
\label{parabolique}
\begin{cases}
\partial_t u -\Delta u =f(u) -\alpha g(y) u \quad
&\text{on } \R\times \R^N\\
u(0,x)=u_0(x) & \text{on } \R^N
\end{cases}
\end{equation}
where $f$ is Fisher-KPP and $u_0$ is an initial condition at least bounded.

\subsection{Extinction for $\alpha\geq \alpha_0$}

Let us fix $\alpha \geq \alpha_0$. We recall that there is no positive
asymptotic profile of \eqref{PA}. 
\begin{theorem}
For $u_0\in L^\infty$, there exists a unique solution $u(t,x)$ of
\eqref{parabolique} and it verifies $u(t,x) \xrightarrow{t\rightarrow
  +\infty} 0$ uniformly for $x\in \R^N$.
\end{theorem}
This section is devoted to the proof of this theorem. 

Let us fix $S=\max(1,\|u_0\|_{\infty})$. Then the constant functions
$0$ and $S$ are respectively sub- and super-solutions of
\eqref{parabolique}. Thus there exists $u(t,x)$ a solution of
\eqref{parabolique} such that $0\leq u\leq S$. By the parabolic
maximum principle, this solution is unique. 

Let us define $w$ the solution of \eqref{parabolique} with the
initial condition $w(0,x)=S$. Since the problem and the initial
condition do not depend on $x_1$, neither does $w$ thus we will write
$w(t,y)$. By the 
maximum principle, $0\leq u\leq w\leq S$ and since $S$ is a
super-solution, $\partial_t w\leq 0$. Thus 
$w(t,y)\xrightarrow{t\rightarrow +\infty} W(y)$ and 
$$
0\leq \limsup_{t\rightarrow +\infty} u \leq W\leq S.
$$ 
Now by parabolic local estimates, $W$ is necessarily solution of 
$$
-\Delta_y W=f(W)-\alpha g(y) W
$$
and thus is a nonnegative asymptotic profile. Since $\alpha \geq
\alpha_0$, $W\equiv 0$. So $u(t,x)$ tends to 0 for $t\rightarrow +\infty$
uniformly in $\R^{N}$. 

\subsection{Spreading for $\alpha < \alpha_0$}

In this section we assume $\alpha<\alpha_0$. So there exists a critical
speed $c^*$ of existence of traveling front for \eqref{FPc}. We assume
that $u_0\in \CC^0_0(\R^N)$, i.e. $u_0$ is continuous and compactly
supported, and that $u_0<V$ where $V$ is the positive asymptotic
profile solution of \eqref{PA}. We will prove the spreading
of the solution of \eqref{parabolique} but we first need the following theorem.

\begin{theorem}
\label{existz}
The unique solution of 
\begin{equation}
\label{uniq}
\begin{cases}
-\Delta z -c \partial_1 z+\alpha g(y) z=f(z) \quad (x_1,y)\in \Rn,\\
0< z(x_1,y) \leq V(y) \quad (x_1,y)\in \Rn.
\end{cases}
\end{equation}  
with $c<c^*$ is $z(x,y) \equiv V(y)$.
\end{theorem}
\begin{proof}
Let us consider the generalized sub-solution with compact support $w(x_1,y)$ defined in \eqref{constrss}. This is possible since $c<c^*$. Up to a decrease of $\ep>0$, we can assume that
$w\leq z$ on $\R^N$. Now by applying the sliding method to $w^\tau$ where $w^\tau(x_1,y)=w(x_1+\tau,y)$ and $z$, one can prove that $w^\tau \leq z$ for all $\tau \in \R$. We can thus define 
$$
\forall y\in \R^{N-1} \quad \underline{z}(y)=\inf_{x_1\in \R} z(x_1,y) \geq 0 
$$ 
and state that $ z\not\equiv 0$.
Now $\underline{z}$ is a super-solution of \eqref{PA} since $\underline{z}=\inf_{h\in \R} z(\cdot+h,\cdot)$ and an infimum of solutions  is a super-solution.

Finally as in section \ref{section: asympt}, we can build a positive sub-solution of \eqref{PA} smaller than $\underline{z}$ and thus by monotone iteration we have a solution of \eqref{PA} between these sub- and super-solution. By uniqueness of the positive solution, we obtain $V\leq z $. And due to condition in \eqref{uniq}, we have $z\equiv V$.
\end{proof}

Let us now turn to the precise study of the spreading of the solution of \eqref{parabolique}
\begin{theorem}
\label{thmcv}
For $u_0\in \CC^0_0(\R^N)$ with $u_0< V$, there exists a unique
solution $u$ of \eqref{parabolique} and 
\begin{eqnarray}
\label{cgrd}
\text{for any }c>c^* \quad \lim_{t\rightarrow +\infty} \sup_{|x_1|\geq ct}
u(t,x)=0, \\
\label{cpt}
\text{for any } c \text{ with } 0\leq c <c^* \quad \lim_{t\rightarrow +\infty} \sup_{|x_1| < ct} |u(t,x)-V(y)|=0.
\end{eqnarray}
\end{theorem}

\begin{proof}
Fix $c>c^*$. Let $U$ denote a traveling front of speed $c^*$. Since
$U(x_1,\cdot)\rightarrow V$ for $x_1\rightarrow -\infty$ locally
uniformly, there exists $L\in \R$ such that $U(x_1-L,y) >u_0(x_1,y)$
for all $(x_1,y)\in \R^N$. Now considering $v(t,x)=U(x_1-L-c^*t,y)$
and applying the comparison principle, we have
$u(t,x)\leq v(t,x)$ for all $t\geq 0$ and $x\in \R^N$.
Thus since $U$ is decreasing in $x_1$
$$
\sup_{|x_1|>ct} u(t,x)\leq \sup_{|x_1|>ct} U(x_1-L-c^* t,y) = \sup_{y\in \R^{N-1}} U((c-c^*)t-L,y).
$$  
Since $c>c^*$ and $U(x_1,y)\xrightarrow{x_1\rightarrow +\infty} 0$ uniformly in $y\in \R^{N-1}$. We see that $\sup_{x_1\geq ct} u(t,x) \rightarrow 0$ as $t\rightarrow +\infty$. Since $u(t,-x_1,y)$ satisfies the same equation \eqref{parabolique}, this shows that $\sup_{x_1\leq -ct}  u(t,x)  \rightarrow 0$ as well as $t\rightarrow +\infty$. Thus \eqref{cgrd}
is proved. 

Assume now $c<c^*$. Let us first prove the following weaken version of \eqref{cpt}:
\begin{lemma}
\label{cvp}
For any $c\in \R$ with $|c|\leq c^*$, 
\begin{equation}
\label{cptw}
\forall (x_1,y)\in \R^N \quad \lim_{t\rightarrow +\infty} |u(t,x_1-ct,y)-V(y)| =0.
\end{equation}
\end{lemma}
\begin{proof}[Proof of lemma \ref{cvp}]
Let us assume that $c\geq 0$, the proof being similar for $c\leq 0$.

Let $v(t,x_1,y)=u(t,x_1-ct,y)$. Then $v$ satisfies the equation
\begin{equation}
\label{pardec}
\partial_t v-\Delta v -c \partial_1 v+\alpha g(y)v=f(v)
\end{equation}
with the initial datum $v(0,x_1,y)=u_0(x,y) \geq 0$ and $\not \equiv
0$. Hence by the parabolic maximum principle, for all $(x_1,y)\in \Rn$
$v(1,x_1,y)>0$. Now since $c< c^*$, in \eqref{constrss}, we constructed $w(x_1,y)\geq 0$ a stationary non-zero
sub-solution of \eqref{pardec} with compact support and $w$ could be
chosen arbitrary small. Hence we can assume $w\leq v(1,\cdot,\cdot)$. So if $\tw$
is the solution of 
$$
\begin{cases}
\partial_t \tw - \Delta \tw -c \partial_1 \tw +\alpha g(y) \tw=f(\tw)
 \quad t>0, \; (x_1,y)\in \Rn\\
\tw(0,x_1,y)=w(x_1,y) \quad (x_1,y)\in \Rn
\end{cases}
$$ 
then by comparison principle, $\forall t\geq 1$ $\forall (x_1,y)\in
\Rn$ $v(t,x_1,y)\geq \tw(t-1,x_1,y)$. Now since $w$ is a sub-solution,
$\tw$ is increasing with respect to $t$ and $0\leq \tw(t,x_1,y) \leq
V(y)$. Therefore, by standard elliptic estimates,
$\tw(t,x_1,y) \xrightarrow{t\rightarrow +\infty} z(x,y)$ and $z$ is a
solution of \eqref{uniq}. By theorem \ref{existz}, we have $z\equiv V$ and   
this complete the proof of the lemma since by the
comparison principle $\tw(t-1,x_1,y)\leq v(t,x_1,y)\leq V(y)$ thus 
$$
\lim_{t\rightarrow +\infty} v(t,x_1,y)=V(y)
$$
which yields \eqref{cptw}.
\end{proof}
Let us now prove \eqref{cpt}, that is the uniform convergence to $V$ in the expanding slab $\{x_1 \leq ct\}$. We will only prove it for $0\leq x_1\leq ct$. Indeed using as before $u(t,-x_1,y)$, the general result follows from the convergence in the set $\{0\leq x_1\leq ct\}$.

Let $c$ with $0<c<c^*$ be fixed and let $\ep>0$ be given (arbitrarily small). For $R>0$ sufficiently large, we know that the principal eigenvalue $\la^R$ of the problem \eqref{vpR} above is such that $\la^R<0$. Denote by $\psi^R>0$ the corresponding eigenfunction of \eqref{vpR}. Under these conditions we know that there exists a unique solution $V^R(y)>0$ of the profile equation in $B_R$ with Dirichlet condition:
\begin{equation}
\label{PAR}
\begin{cases}
-\Delta V^R+\alpha g(y) V^R=f(V^R) \qquad \text{in } B_R \\
V^R=0 \quad \text{on } \partial B_R, \qquad V^R>0 \quad \text{in } B_R. 
\end{cases}
\end{equation}  
(Compare e.g. \cite{HB81}). Moreover, it is straightforward  to show that $V^R$ is increasing with $R$ and that $\lim_{R\rightarrow +\infty} V^R(y)=V(y)$.

Let us choose $R>0$ sufficiently large so that 
for all $y\not\in B_R$ $V(y)<\ep $ and for all $y\in \overline{B_R}$ $0<V(y)-V^R(y)<\ep$. The proof of the uniform convergence to $V$ for $c<c^*$ will rest on the following Proposition.
\begin{proposition}
\label{existv}
Let $c$ be such that $0<c<c^*$. Then, with $R$ chosen as above, there exists a solution $v_c(x_1,y)$ defined for $x_1\in \R^-$, $y\in \overline{B_R}$ of equation
\begin{equation}
\label{Pv}
-\Delta v-c\partial_1 v +\alpha g(y)v=f(v) \quad x_1\leq 0, \; y\in \overline{B_R}
\end{equation}
satisfying the following properties:
$$
\begin{cases}
v_c>0 \text{ and } \partial_1 v_c<0 \text{ in } \R_*^- \times B_R, \\
v_c(0,y)=0 \text{ for } y\in \overline{B_R}, \\
v_c(x_1,y)=0 \text{ for } y\in \partial B_R, \; x_1\leq 0, \\
v_c(-\infty,y)=V^R(y)  \text{ for } y\in \overline{B_R}.
\end{cases}
$$ 
\end{proposition}
Postponing the proof of this proposition, let us complete the proof of Theorem \ref{thmcv}. Extending $v_c$ by $0$ for $x_1\geq 0$ turns $v_c$ into a (generalized) sub-solution of equation \eqref{Pv} in the cylinder $\R \times B_R$ (see \cite{BL80}).
Therefore $v_c(x_1-c(t-t_0),y)$ is a sub-solution of the equation \eqref{parabolique} in this cylinder for all $t_0 \geq 0$ and all $c\in (0,c^*)$.

By Lemma \ref{cvp} (applied here in the case $c=0$), we can fix $t_0>0$ sufficiently large such that for $t\geq t_0$ we have
$$
u(t,0,y)\geq V(y)-\frac{\delta}{2} \quad \text{for all } y\in \overline{B_R}
$$ 
where $\displaystyle \delta =\min_{\overline{B_R}} (V-V^R)>0$. Therefore,
$$
u(t,0,y)>V^R(y) \quad \text{for all } t\geq t_0 \text{ and all } y\in \overline{B_R}.
$$

We fix $\tilde{c} \in (c,c^*)$ and we consider $v(t,x_1,y)=v_{\tilde{c}}(x_1-\tilde{c}(t-t_0),y)$. 
In the region $D=(0,+\infty)\times B_R$, $u$ is a solution and $v$ a sub-solution of equation \eqref{parabolique} and for any time $t\geq t_0$
$$
u(t,x_1,y) \geq v (t,x_1,y) \quad \text{for } (x_1,y)\in \partial D.
$$
Moreover, $u(t_0,x_1,y)\geq v(t_0,x_1,y)=0$ in $D$. The comparison principle then yields
$$
u(t,x_1,y)\geq v(t,x_1,y) \quad \text{in } D.
$$
Therefore 
\begin{eqnarray*}
\limsup_{t\rightarrow +\infty}  \sup_{0\leq x_1\leq ct \atop y\in B_R} \big(V(y)-u(t,x)\big) &\leq& \limsup_{t\rightarrow +\infty}  \sup_{0\leq x_1\leq ct \atop y\in B_R} \big(V(y)- v_{\tilde{c}}(x_1-\tilde{c}(t-t_0),y) \big) \\
&\leq& \limsup_{t\rightarrow +\infty} \sup_{y\in B_R} \big( v_{\tilde{c}}((c-\tilde{c})t+\tilde{c}t_0,y) \big) \\
&\leq& \sup_{y\in B_R} \big( V(y) -V^R(y)\big)<\ep.
\end{eqnarray*}
Outside of $B_R$ we already know that $0<u<V<\ep$ for any $t\geq 0$, $x_1 \in \R$ and $|y|\geq R$. Therefore
$$
\limsup_{t\rightarrow +\infty}  \sup_{0\leq x_1\leq ct} \big(V(y)-u(t,x)\big) \leq \ep 
$$
Since this is true for all $\ep>0$ (and for $-ct\leq x_1\leq 0$), we have thereby established \eqref{cpt}.

It now remains to prove Proposition \ref{existv} which we carry now.
As in \eqref{constrss}, we construct a sub-solution of the equation \eqref{Pv} with compact support, namely:
 $$
w(x_1,y)= \begin{cases}
\ep e^{\sigma x_1} \cos(\frac{\pi}{2L}x_1+\frac{\pi}{2}) \psi^R(y)  \text{ if }
-2L<x_1<0, \; y\in B_R, \\
0  \text{ otherwise.}
\end{cases}
$$
In comparison with \eqref{constrss}, there is a translation in $x_1$ such that the support of $w$ now lies in $\R_- \times B_R$. 

For any $b<0$, let $z_b$ be the solution of 
$$
\begin{cases}
-\Delta z_b-c\partial_1 z_b +\alpha g(y)z_b=f(z_b) \quad \text{in } (b,0)\times B_R, \\
z_b(b,y)=V^R(y), \quad z_b(0,y)=0 \quad \text{for } y\in \overline{B_R} \\
z_b(x_1,y)=0 \quad  \text{for } x_1\in (b,0), \; |y|=R.
\end{cases}
$$
Since $V^R$ is a super-solution and $0$ a sub-solution, there exists a solution of this problem. By the sliding method of \cite{BN91}, we know that this solution is unique and satisfies $\partial_1 z_b<0$ in $(b,0) \times B_R$. 

Next, for $b<-L$, wince $w$ is a sub-solution, we also know that 
$$
\forall b\leq -L \quad \forall (x_1,y)\in (b,0) \times \overline{B} \quad z_b(x_1,y)>w(x_1,y).
$$
This allows us to pass to the limit when $b\rightarrow -\infty$. Clearly $z_b(x_1,y) \xrightarrow{b\rightarrow -\infty} v_c(x_1,y)$. By the lower bound, $v_c(x_1,y)>w(x_1,y)$ which shows that $v_c(x_1,y)>0$ in $\R_-^*\times B_R$. Since $\partial_1 v_c \leq 0$ and $v_c\not\equiv 0$, we also know that  $\partial_1 v_c < 0$ in $ \R_-^*\times B_R$. Now since $\lim_{x_1\rightarrow -\infty} v_c(x_1,y)$ must be a positive solution of \eqref{PAR}. Hence by uniqueness we get $v_c(-\infty,y)=V^R(y)$. This completes the proof of Proposition \ref{existv} and therefore of Theorem \ref{thmcv} 
\end{proof}


\section{The case of a positive non-linearity. }
\label{section: positive}

In this section, we prove Theorem \ref{existpos} about the existence of traveling fronts in the positive case. We use the notations of the preceding sections, in particular $\la$ still denotes the principal eigenvalue of the linearized operator around 0 and $\pa$ an associated eigenfunction. We are interested in a traveling front solution of \eqref{EP} when $f$ is only assume to be of the positive type, that is $f:\R\to \R$ is $\CC^1$ with 
\begin{equation}
\label{hyppos}
f(0)=f(1)=0, \; f>0 \text{ on } (0,1) \text{ and } f'(0)>0.
\end{equation}

\subsection{Asymptotic profiles}

The linearized operator around 0 is exactly the same as in the Fisher-KPP case thus there exists $\alpha_0>0$ such that $\la<0$ for $\alpha<\alpha_0$ and $\la\geq 0$ for  $\alpha \geq \alpha_0$. In the same way as before, we can prove the existence result:
\begin{proposition}
For $\alpha<\alpha_0$, there exists $V(y)$ a maximal positive asymptotic profile solution of \eqref{PA}.
\end{proposition} 
However, in this more general case, we have no information on the uniqueness of the positive asymptotic profile nor on the non-existence of profiles for $\alpha\geq \alpha_0$. Actually, this will depend on the non-linearity $f$. 
\begin{proof}
As in Theorem \ref{existPA}, $\ep \pa$ is a subsolution for $\ep>0$ small enough and 1 is a supersolution. Using monotone iterations, we can construct a maximal positive asymptotic profile.
\end{proof}

Since the positive asymptotic profile may not be unique, we will need the following lemma before turning to the construction of traveling fronts.
\begin{lemma}
\label{thetaa}
For any $\alpha<\alpha_0$, there exists $\theta_\alpha>0$ such that any positive asymptotic profile $W(y)$ solution of \eqref{PA} satisfies
 \begin{equation}
\label{Wtheta}
W(0)\geq 2\theta_\alpha.
\end{equation} 
\end{lemma}
\begin{proof}
By contradiction, assume that there exist $W_n>0$ solution of \eqref{PA} with $W_n(0) \xrightarrow{n\to \infty} 0$. Then $\psi_n=\frac{W_n}{\|W_n\|_\infty}$ is a solution of
$$
-\Delta \psi_n+\alpha g(y)\psi_n=\frac{\|W_n\|_\infty \psi_n}{\|W_n\|_\infty}.
$$ 
Since $\psi_n$ is bounded, up to extraction of a subsequence, we can let $n$ tend to $\infty$ to obtain $\psi_n \to \psi_\infty\geq 0$ with $\|\psi_\infty\|_\infty=1$ and 
$$
-\Delta \psi_\infty +\alpha g(y)\psi_\infty=f'(0)\psi_\infty.
$$
In the previous limit, we made use of the compactness argument derived from the fact that $\psi_n(y)\rightarrow 0$ as $|y|\to \infty$ uniformly in $n$ as is obtained from Theorem \ref{BR}.
Thus, $\psi_\infty$ is a principal eigenfunction associated with the eigenvalue 0 which contradicts the fact that $\la<0$. 
\end{proof}

\subsection{Existence of traveling fronts}

In this section, we will use the same method as in section \ref{section: cetoile} to construct a traveling front solution of \eqref{PA}. However due to the possible non-uniqueness of the positive asymptotic profile, the result will be somewhat weaker in that the limiting profile is not prescribed. More precisely, we will prove the existence of $c\in \R$ and $u$ solutions of 
\begin{equation}
\label{FPcw}
\begin{cases}
-\Delta u - c\partial_1 u + \alpha g(y) u=f(u), \quad  x=(x_1,y)\in
\Rn \\
u(x_1,.)\xrightarrow{x_1\rightarrow +\infty} 0 \; \text{uniformly in }
y\in \R^{N-1},\\
u> 0, \quad u \text{ bounded} \quad \text{and } \partial_1u<0.
\end{cases}
\end{equation}  

The construction of the solution follows the same line as before. We start by solving the problem on a domain bounded in $x_1$, precisely we study solution of \eqref{FPB} where $V$ is the maximal positive asymptotic profile. Since $V$ is maximal, Theorem \ref{boite} still holds true. The only difficulty is to translate Theorem \ref{ca} to the case of a positive non-linearity. Having this aim in mind, we introduce two notations: 
$$
m=\sup_{(0,1]} \frac{f(s)}s \text{ and }\mu_\alpha =\la+f'(0)-m \leq \la<0.
$$
In the Fisher-KPP case, we observe that $m=f'(0)$ and thus $\mu_\alpha=\la$. We will prove the following result.
\begin{theorem}
\label{capos}
Let $\ep>0$ be fixed such that $\ep \pa<V$ (see previous section). Let $\delta>0$ be such that  $\delta <-\la <f'(0)$ and let $\eta>0$ be such that $\eta<\ep$ and $\forall s\in [0,\eta] \; f(s)\geq (f'(0)-\delta)s$. We fix $\theta \in (0, \frac \eta 2)$ such that $\theta<\theta_\alpha$ (see Lemma \ref{thetaa} for definition of $\theta_\alpha$). Then there exists $A_\ep>0$ such that for all $a\geq A_\ep$ there exists a unique speed $ca \in (0, 2\sqrt{-\mu_\alpha}+\ep)$ such that $u_a^{c_a}$ the solution of \eqref{FPB} satisfies $u_a^{c_a}(0,0)=\theta$.
\end{theorem}
\begin{proof}
The only difference with the proof of Theorem \ref{ca} is in the upper bound of $c_a$ (section ii) in the proof of Theorem \ref{ca}). It goes as before but we need to replace $\la$ by $\mu_\alpha$.

So let us construct an explicit super-solution of $\eqref{FPB}$ for $c=2\sqrt{-\mu_\alpha}+\ep$. As before, we can fix $\beta<\alpha$ such that $2\sqrt{\mu_\beta}\leq 2\sqrt{\mu_\alpha}+\ep$ and we consider $\psi_\beta$ the positive eigenfunction of the linearized operator around 0 with the normalization $\|\psi_\beta\|_\infty=1$. We fix $R>0$ such that for all $r\geq R$
$$
(\alpha-\beta)g(r)+\mu_\beta>0 \text{ and } \alpha g(r)>m. 
$$
Let us then fix $k>0$ such that $k \psi_\beta\geq V$ on $\bar{B}_R$. Then $k\psi_\beta\geq V$ on $\R^{N-1}$. Indeed we argue by contradiction and assume that $\min_{\R^{N-1}} (k\psi_\beta -V)=(k\psi_\beta-V)(y_0)<0$ but at this point $y_0\in \R^{N-1}\setminus \bar{B}_R$, we have
\begin{eqnarray*}
-\Delta (k\psi_\beta-V)(y_0)+(\alpha g(y_0)-m)(k\psi-V)(y_0) = \qquad \qquad \qquad  \\
\hfill \mu_\beta k\psi_\beta(y_0)+(\alpha-\beta)g(y_0)k\psi_\beta(y_0)+mV(y_0)-f(V(y_0))>0
\end{eqnarray*} 
and this yields a contradiction. 

Then we prove as in Theorem \ref{ca} that $w(x_1,y)=z(x_1)k\psi_\beta(y)$ is a super-solution if $z$ is a solution of 
$$
\begin{cases}
z''+cz'-\mu_\beta z=0 \quad \text{on } (-a,a), \\
z(-a)=1, \quad z(a)=0
\end{cases}
$$ 
and as before we obtain that 
$$
0<u_a^c(0,0)<e^{-\frac c2 a} k\leq e^{-a\sqrt{-\mu_\alpha}}k
$$
and the upper bound of $c_a$ for large $a$ is thus proved. 
\end{proof}
Then the convergence of $u_a^{c_a}$ to a solution of \eqref{FPcw} when $a$ tends to $+\infty$ is exactly the same except that the non-uniqueness of the positive asymptotic profile prevents us from determining the precise limit of $u(x_1,.)$ for $x_1 \rightarrow -\infty$. 

We leave it as an open problem to know whether there always is a traveling front connecting the maximum profile $V(y)$ to the 0 solution.


\section{The case of a bistable non-linearity. Asymptotic profiles.}
\label{section: APbist}

In this section we consider again equation \eqref{main} but in the bistable framework. That is, we assume that $f$ is a $\CC^1$ function that satisfies the following assumptions for some $\theta\in(0,1)$:
\begin{equation}
\label{fbistable}
f(0)=f(\theta)=f(1)=0, \quad f(s)<0 \text{ for } s\in (0,\theta) \text{ and } f(s)>0 \text{ in } (\theta,1),
\end{equation}
\begin{equation}
\label{derivf}
f'(0)>0, \; f'(1)>0.
\end{equation}
We also assume that
\begin{equation}
\label{fcpos}
\int_0^1 f(s)ds>0.
\end{equation}
We are concerned here with the existence of traveling front solutions of \eqref{main}, that is, $(c,u)$ solution of \eqref{FPc}. First we require some preliminary results on the equation \eqref{PA} in the bistable case.

\subsection{Existence of asymptotic profiles in the bistable case}

Consider equation 
\begin{equation}
\label{PAb}
\begin{cases}
\Delta u +f(u)-\alpha g(y)u=0, \quad y\in \R^{N-1}, \\
u\geq 0, \quad u \text{ bounded},
\end{cases}
\end{equation}
under the same assumption \eqref{Hypgpos} and \eqref{Hypginf} as above for the function $g$.

The existence of solutions depends on $\alpha$ and is obtained in the following theorem.
\begin{theorem}
\label{thmPAb}
Let $f$ and $g$ satisfy the above assumptions. There exists a threshold value $\alpha^*\in (0,\infty)$, such that:
\begin{description}
\item[i)] For any $\alpha \in ( \alpha^*,+\infty)$, \eqref{PAb} does not have any positive (non-zero) solution.
\item[ii)] For any $\alpha \in (0, \alpha^*]$, \eqref{PAb} admits a maximal positive solution $V(y)$.
\item[iii)] For any $\alpha \in  (0, \alpha^*)$, \eqref{PAb} admits a second positive solution $W(y)$ with $0<W(y)<V(y)$.
\end{description}
\end{theorem}
The rest of this section is devoted to the proof of this Theorem.

This Theorem follows from the observation that for $\alpha>0$ any positive solution $u(y)$ of \eqref{PAb} satisfies $u(y)\rightarrow 0$ as $|y|\rightarrow \infty$. This is obtained from Corollary \ref{decry}. 

Next, by the maximum principle, any solution of \eqref{PAb} satisfies $0\leq u\leq 1$ (we think of $f(s)$ as having been extended by 0 outside $[0,1]$).

Now $\us\equiv 1$ is a super-solution of problem \eqref{PAb}. Any solution of \eqref{PAb} for $\alpha$ is a sub-solution of \eqref{PAb} for any parameter $\beta\leq \alpha$. Therefore, if there exists a positive bounded solution of \eqref{PAb} for $\alpha$, there also exists a positive solution for any $0<\beta\leq \alpha$.

Next, we claim that for small enough $\alpha>0$, \eqref{PAb} admits a positive solution. Indeed, consider the functional defined on $\HH$:
$$
J(w)=J_\alpha (w)=\int_{\R^{N-1}} \left( \frac{1}{2} |\nabla w|^2 +\frac{\alpha}{2} g(y) w^2 -F(w) \right)  dy
$$
where $F(z)=\int_0^z f(s)ds$. Recall that $f$ is extended by 0 outside $[0,1]$, thus $F$ is bounded. Since $g(r)\rightarrow \infty$ as $r\rightarrow \infty$, it is straightforward to show that there exists a minimizer $v$ of $J(w)$:
$J(v)=\min\{ J(w), \, w\in \HH \}$. Furthermore, we know that $v\geq 0$ and $v$ is a solution of \eqref{PAb} (see Theorem \ref{energy} for details).

Let us show that for $\alpha>0$ small enough $J(v)<0$. To this end, let $\zeta_R$ be defined by
$$
\zeta_R(y)=
\begin{cases}
1 &\text{if } |y|\leq R \\
R+1 - |y| &\text{if } R \leq |y| \leq R+1 \\
0 & \text{if } |y|\geq R+1
\end{cases}
$$ 
Then $\zeta_R \in \HH$ and 
$$
J_0(\zeta_R)=\int_{\R^{N-1}} \frac{|\nabla \zeta_R|^2}{2} -F(\zeta_R) \leq -F(1) |B_R|+C|B_{R+1} \setminus B_R|
$$
where $|A|$ denotes the volume of $A$ and $C$ is a constant.
Since $-F(1)<0$ by \eqref{fcpos}, we see that by choosing $R$ large enough, $J_0(\zeta_R)<0$. Then for such an $R$ fixed, we see that $J_\alpha(\zeta_R)<0$ provided $\alpha >0$ is small enough. This guarantees that $J_\alpha(v)<0$.

It follows that $v\not\equiv 0$. By the maximum principle, we then have $0<v<1$. This shows that for small $\alpha>0$, \eqref{PAb} admits a positive solution.

\medskip

Next, we show that if $\alpha$ is large enough \eqref{PAb} does not admit any positive solution. This can be seen by multiplying the equation by $u$ and integrating to yield:
\begin{equation}
\label{PAint}
\int |\nabla u|^2+\alpha \int g(y) u^2 = \int f(u) u \leq m\int u^2
\end{equation}
where $\displaystyle m=\sup_{s>0} \frac{f(s)}{s}>0$. We conclude with the following lemma:
\begin{lemma}
\label{CSH}
Under the assumption \eqref{Hypginf} $g(r)\xrightarrow{r\rightarrow \infty} \infty$, for any $\ep>0$, there exists a constant $K(\ep)>0$ such that for all $u\in \HH$ one has:
$$
\int_{\R^{N-1}} u^2 \leq \ep \int_{\R^{N-1}}|\nabla u|^2 +K(\ep) \int_{\R^{N-1}} g(y) u^2 .
$$
\end{lemma}
Indeed choosing in the lemma $\ep=\frac{1}{2m}$, we get from \eqref{PAint}
$$
\frac{1}{2} \int |\nabla u|^2+\left( \alpha -m K(\frac{1}{2m})\right)  \int g(y) u^2 \leq 0
$$
This shows that for $\alpha \geq mK(\frac{1}{2m})$, the only solution of $\eqref{PAb}$ is $u\equiv 0$.
\begin{proof}[Proof of Lemma \ref{CSH}]
Let $\delta=\delta(\ep)>0$ be chosen such that the principal eigenvalue of $-\Delta$ in $H ^1_0(B_{2\delta})$ is larger than $\frac{4}{\ep}$. Let $\chi$ be a smooth cutoff function such that $\chi(r)=1$ if $0\leq r\leq \delta$, $\chi(r)=0$ if $r\geq 2\delta$ and $0\leq \chi \leq 1$. Consider $u_1=\chi u$ and $u_2=(1-\chi)u$ so that $u=u_1+u_2$. Using $(a+b)^2 \leq 2\left( a ^2 +b ^2\right)$, since $u_1\in H^1_0(B_{2\delta})$ by Poincar\'e's inequality, we have
$$
\int_{R^{N-1}} {u_1}^2=\int _{B_{2\delta}} {u_1}^2\leq \frac{\ep}{4} \int _{B_{2\delta}} |\nabla u_1|^2 \leq \frac{\ep}{2} \left( \int _{B_{2\delta}} |\nabla u|^2 \chi ^2 + \int _{B_{2\delta}\setminus B_\delta} u^2 |\nabla \chi|^2 \right).
$$
So that 
$$
\int_{R^{N-1}} {u_1}^2 \leq  \frac{\ep}{2} \int _{B_{2\delta}} |\nabla u|^2 +\ep k_1(\ep) \int_{|y|\geq \delta } u^2
$$
where $k_1(\ep)\geq |\nabla \chi|^2$.

Next $\displaystyle \int _{R^{N-1}} {u_2}^2 \leq  \int_{|y|\geq \delta } u^2$.
Therefore 
$$
\int _{R^{N-1}} {u}^2 \leq 2\left( \int {u_1}^2+\int {u_2}^2 \right) \leq \ep \int  |\nabla u|^2+ K(\ep) \int g(y)u^2
$$
where $K(\ep)=2\frac{\ep k_1(\ep)+1}{g(\delta)}$. The lemma is thus proved.
\end{proof}

The next step is to prove that the set of $\alpha >0$ such that \eqref{PAb} has a solution is a closed set. Let $\alpha_j \rightarrow \alpha^*$ be a sequence such that \eqref{PAb} admits a solution $u_j$ such that $0<u_j<1$ for all $j$. Note that by the maximum principle, $\theta <\max u_j <1$. The sequence $(u_j)$ is bounded by 1 and by standard elliptic estimates is locally compact. Therefore, one can extract a subsequence $u_j$ such that $u_j\rightarrow u^*$ uniformly on compact sets in the $\CC^2$-norm. Therefore $u^*$ is a solution of \eqref{PAb} for the value $\alpha=\alpha^*$. We know that $u^*\geq 0$, but since $\max u_j >\theta$, we see that $\max u^*\geq \theta$. Indeed by section \ref{section: prelim}, $u_j(y)\rightarrow 0$ as $|y| \rightarrow \infty$ uniformly with respect to $j$. Therefore $u^*>0$ and \eqref{PAb} also has a positive solution for $\alpha^*$. This shows that the set of $\alpha$ such that \eqref{PAb} has a positive solution is an interval $(0,\alpha^*]$ with $0<\alpha^*<\infty$.

Considering the evolution equation 
$$
\begin{cases}
\partial_t z-\Delta z +\alpha g(y) z=f(z), \quad t>0, \; y\in \R^{N-1}, \\
z(0,y)=1,
\end{cases}
$$
we see that $t\mapsto z(t,y)\geq 0$ is decreasing and therefore has a limit. This limit is necessarily the maximal positive solution $V=V_\alpha$ for the  $\alpha$ for which \eqref{PAb} has a positive solution, that is $\alpha \in (0,\alpha^*]$, or is 0 in the opposite case, that is when $\alpha >\alpha^*$.

The existence of a second solution when $0<\alpha<\alpha^*$ is inspired from a work of P. Rabinowitz \cite{PR1}.
In a slightly different formulation, the existence of pairs of solutions is established in \cite{PR1} by a topological degree argument for bistable type nonlinearities and another type of parameter dependance. The use of the topological degree involves compact operators and the results of \cite{PR1} are set in the framework of bounded domains. A similar construction can be carried here owing to the condition \eqref{Hypginf} $g(r)\rightarrow +\infty$ as $r\rightarrow +\infty$. Indeed, under this condition, the injection $\HH \hookrightarrow L^2(\R^{N-1})$ is compact.This allows one to construct a compact operator and to carry the argument of \cite{PR1} to the present framework.

Since we will not use the second solution, we will leave out the details of the proof of the existence of a second solution.

\subsection{Stable asymptotic profiles}

As we have seen, a solution of \eqref{PAb} is obtained by the minimization of $J=J_\alpha$ defined above. The proof of the existence of the previous solution for $\alpha>0$ small yields the following result.
\begin{proposition}
\label{PAminJ}
There exists $0<\alpha_*\leq \alpha^*$ such that  for all $\alpha \in (0,\alpha_*)$ there exists a minimum $v_\alpha>0$ of $J_\alpha$ and such that 
$$
J_\alpha(v_\alpha)=\min_{H^1(\R^{N-1})} J_\alpha <0.
$$
\end{proposition}
 
In the following, we require the notion of stable solution. 
\begin{definition}
\label{defvp}
Let $v$ be a solution of \eqref{PAb}. Eigenvalues of the linearized problem about $v$ are defined as the eigenvalues $\lambda$ of 
$$
 -\Delta \varphi +\alpha g(y) \varphi -f'(v)\varphi = \lambda \varphi \quad \text{in } \R^{N-1}.
$$
The principle eigenvalue is uniquely determined by the existence of a corresponding eigenfunction $\varphi$ with $\varphi>0$. 
We say that $v$ is (weakly) stable if the principal eigenvalue $\lambda=\lambda_1[v]$ of the linearized problem satisfies $\lambda_1[v]\geq 0$.
\end{definition}

It is well known that the maximal solution $V(y)$ given by Theorem \ref{thmPAb} when $0<\alpha\leq \alpha^*$ is weakly stable. Likewise, the minimum solution of the energy of the Proposition \ref{PAminJ} above, when $0<\alpha<\alpha_*$, is a weakly stable solution. 

In the following we consider the case $0<\alpha<\alpha_*$ and we make the following assumption.
\begin{equation}
\label{uniqsPA}
\text{There exists a unique positive stable solution of \eqref{PA}.} 
\end{equation}
This condition implies that the minimizer solution $v_\alpha$: $J_\alpha(v_\alpha)=\min\{J_\alpha(v), \; v\in \HH \}$ coincides with the maximum solution $V$.

We leave it as an open problem to give sufficient conditions for the uniqueness of the stable solution. Uniqueness results have been given for analogous problems but with $\alpha=0$, which would rather correspond to the minimal solution in our framework \cite{PS}.
Likewise it would be interesting to give sufficient conditions that ensure that $\alpha_*=\alpha^*$.
Condition \eqref{uniqsPA} has several implications that we can state.
\begin{proposition}
For $\alpha \in (0,\alpha_*]$, under condition \eqref{uniqsPA}, there does not exist a pair of distinct ordered functions $(v_1,v_2)$ with $0<v_1\leq v_2 < V$, $v_1$ is a sub-solution and $v_2$ is a non-maximal solution. That is, if $0<v_1\leq v_2<V$ are respectively sub-solution and solution of \eqref{PAb}, then $v_1\equiv v_2$.
\end{proposition}
\begin{proof}
The proof follows the observation in \cite{BN}. However, it requires new elements in view of the unbounded domain. If $v_1<v_2$, let $\varphi_2$ be a principal eigenfunction of the linearized problem corresponding to $\lambda_1[v_2]$. Since $0$ and $V$ are the only stable solutions, $\lambda_1[v_2]<0$. We claim that for $\ep >0$ sufficiently small, $\vm=v_1-\ep \varphi_2$ is a super-solution of \eqref{PAb}. Indeed
\begin{eqnarray*}
-\Delta \vm +\alpha g(y) \vm -f(\vm)&=&f(v_2)-f(\vm)-f'(v_2)\ep \varphi_2 -\lambda_1[v_2]\ep \varphi_2\\
&=& \left( \frac{f(v_2)-f(v_2-\ep \varphi_2) }{\ep \varphi_2} -f'(v_2)-\lambda_1[v_2] \right) \ep \varphi_2.
\end{eqnarray*}
The right hand side is positive if $\ep>0$ is sufficiently small.

Next, given $R>0$, we can choose $\ep>0$ small enough so that $v_1<v_2-\ep \varphi_2 $ in $\overline{B_R}$. We choose $R$ so that $v_1(y) \leq \delta$ for all $|y|\geq R$ and $f$ is decreasing on $[0,\delta]$. We claim that then $v_1\leq v_2-\ep \varphi_2$ in $\R^{N-1}\setminus \overline{B_R}$. Argue by contradiction. In this were not the case, then, since $v_1$, $v_2$ and $\varphi_2$ converge to 0 at infinity, there exists $y$, $|y|>R$ such that 
$$
\min_{\R^{N-1}} \{ v_2-\ep \varphi_2-v_1\}=v_2(y)-\ep \varphi_2(y)-v_1(y) <0
$$
This implies that $0< \vm(y)<v_1(y)\leq \delta$. Denote $L$ the operator $L=-\Delta +\alpha g(y)$. Since $0\leq L(\vm-v_1)-(f(\vm)-f(v_1))$ and $f(\vm(y))-f(v_1(y))>0$, at the point $y$ we get $L(\vm-v_1)(y)> 0$. Therefore, we have reached a contradiction. This shows that $v_1\leq v_2-\ep \varphi_2$. Now we have a super-solution $\vm$ above a sub-solution $v_1$. This implies that there exists a stable solution $v$ such that $v_1 \leq v \leq v_2-\ep\varphi_2 <V$. This however is in contradiction with condition \eqref{uniqsPA}.
\end{proof}

From this property, we derive the following useful consequence.
\begin{proposition}
Let $\alpha \in (0,\alpha_*)$ and let $W$ be the maximal solution of equation \eqref{PAb} with the value $\alpha_*$ of the parameter. Then, any other solution $v$ of \eqref{PAb} with parameter $\alpha$ that is not the maximal solution cannot be above $W$.
\end{proposition}
This immediately follows from the previous proposition as $W$ is a sub-solution of the equation for the value $\alpha<\alpha_*$ and $W<V$. 

A consequence of this proposition is 
\begin{proposition}
\label{Visole}
For $\alpha \in (0,\alpha_*]$ and under condition \eqref{uniqsPA}, the maximal solution $V$ is isolated in $L^\infty$ topology. Therefore, there exists $\theta_1>\theta$ such that if $v$ is a solution of \eqref{PAb} with $v(0)\geq \theta_1$ then $v\equiv V$.
\end{proposition}
As we have done before,we can prove that if $v$ is a solution such that $v\geq W$ in $\overline{B_R}$, then $v\geq W<V$ in $\R^{N-1}$. Then any solution $v\not \equiv V$ is such that there exists $y \in \overline{B_R}$ such that $v(y)\leq W(y)$ and therefore $\|v-V\|_{L^\infty} \geq \min_{\overline{B_R}}V-W=\delta>0$. 

Now, if there exist a sequence $v_n$ of solutions of \eqref{PAb} such that $v_n(0)\rightarrow V(0)$ then by elliptic estimates $v_n\rightarrow W$ a positive solution of \eqref{PAb} and $W(0)=V(0)$ so $W\equiv V$ by the maximum principle which contradicts the fact that $V$ is isolated. 
  

\section{Traveling fronts for a bistable non-linearity}
\label{section: TFbist}

In this section we assume that $f$ if of bistable type and satisfies \eqref{fbistable}-\eqref{fcpos}. In addition, we assume that $0<\alpha<\alpha_*$ and that condition \eqref{uniqsPA} is fulfilled. Therefore, there exists a unique non-zero stable solution $V(y)=V_\alpha(y)$ of the profile equation \eqref{PAb}. Therefore $V>0$, $J_\alpha(V)=\min \{ J(w), \; w\in \HH \}$, $\lambda_1[V] \geq 0$ and $V$ is isolated in the $L^\infty$ topology. Furthermore $V$ is the maximal solution. Any other non-zero solution $w$ satisfies $0<w<V$ in $\R^{N-1}$ and $\lambda_1[w]<0$ where $\lambda_1[w]$ is the principal eigenvalue of the linearized problem defined in definition \ref{defvp}.

In this section, we prove the existence of a traveling front solution of \eqref{main} representing an invasion of $0$ by the state $V$ at positive speed. Such a solution is given as a pair $(c,u)$ of 
\begin{equation}
\label{FPb}
\begin{cases}
-\Delta u -c \partial_1 u+\alpha g(y)u=f(u) \quad \text{in } \R^N \\
u(-\infty,y)=V(y), \quad u(+\infty,y)=0   
\end{cases}
\end{equation}
with $c<0$ and $u:\R^N\rightarrow (0,1)$.

We follow the construction of a solution given above. Namely, let $a\geq 1$ and in the slab $\Sigma_a=(-a,a)\times \R^{N-1}$, consider the problem
\begin{equation}
\label{FPab}
\begin{cases}
-\Delta u -c\partial_1 u +\alpha g(y) u=f(u) \quad \text{in } \Sigma_a,\\
u(-a,y)=V(y), \quad u(+a,y)=0.   
\end{cases}
\end{equation}
We recall that for any $c\in \R$, for $a$ fixed, there exists a unique solution $u=u^c$ of \eqref{FPab}. Furthermore, $0<u<V$ and $\partial_1 u <0$ in $\Sigma_a$.The mapping $c\mapsto u^c$ is decreasing. 

Up to here, the procedure is the same as before. From this point on however, we need to modify the above argument since we used the fact that $f$ was positive. 

Our first task is to prove the following
\begin{proposition}
There exists a unique $(c_a,u_a)$ such that $u_a$ is a solution of \eqref{FPab} for speed $c_a$ and $u_a$ satisfies the normalization condition 
\begin{equation}
\label{NC}
\max_{y\in R^{N-1}} u_a(0,y)=\theta.
\end{equation}
Let us first prove the existence of $c_a$. The uniqueness is clear.

The parameter $c_a$ is bounded independently of $a\geq 1$. Moreover, 
$$
\liminf_{a\rightarrow +\infty} c_a \geq 0.
$$
\end{proposition} 
\begin{proof}
The bound from above is obtained simply by comparison with the one dimensional problem. Indeed, consider the ODE problem for $z=z(x_1)$:
\begin{equation}
\label{TF1D}
\begin{cases}
-z''-\gamma z'=f(z) \quad \text{in } (-a,a) \\
z(-a)=1, \quad z(+a)=0, \quad z(0)=\theta
\end{cases}
\end{equation}
It is known that there exists a unique value $\gamma^a$ for which \eqref{TF1D} has a (unique) solution $z$. Furthermore, $\lim_{a\rightarrow +\infty} \gamma^a=\gamma^*$ where $\gamma^*$ is the unique speed of traveling fronts for the 1D equation
\begin{equation*}
\begin{cases}
-z''-\gamma^* z'=f(z) \quad \text{in } (-a,a) \\
z(-\infty)=1, \quad z(+\infty)=0
\end{cases}
\end{equation*}
Comparing \eqref{TF1D} with \eqref{FPab}, we see that for each $c=\gamma^a$, the solution $z$ of \eqref{TF1D} is a super-solution of \eqref{FPab}, thus $z>u^{\gamma^a}$ and for all $y\in\R^{N-1}$, $u^{\gamma^a}(0,y)<z(0)=\theta$. Since $c\mapsto u^c$ is decreasing, we see that 
$$
\max_{y\in R^{N-1}}u^c(0,y) <\theta \text{ for all }c\geq \gamma_a.
$$
Assume now that $\max_{y\in R^{N-1}}u^c(0,y) <\theta$ for all $c\in \R$. Passing to the limit for $c\rightarrow -\infty$, $u^c$ converges toward a positive solution $v$ of \eqref{PAb} with $\max v <\theta$. By the maximum principle, it is impossible thus there exists a unique $c_a \in (-\infty,\gamma^a)$ such that \eqref{NC} is fulfilled.

Since $\gamma^a \rightarrow \gamma^*<\infty$ as $a\rightarrow +\infty$ and $a \mapsto  \gamma^a$ is a continuous function, this shows that $\displaystyle \sup_{a \geq 1} c_a <\infty$. 

Since $a \mapsto c_a$ is continuous, in order to complete the proof of the Proposition, it suffices to show that $\liminf_{a\rightarrow \infty} c_a \geq 0$. For this, we argue by contradiction and assume that for a sequence $a_j\rightarrow +\infty$ there holds $c_{a_j}<0$. For the sake of simplicity, we write $a$ instead of $a_j$. Since $c\mapsto u^c$ is decreasing, from this we infer that along this subsequence, the solution $v=v^a$ of 
$$
\begin{cases}
-\Delta v+\alpha g(y) v =f(v) \quad \text{in } \Sigma_a \\
v(-a,y)=V(y), \quad v(+a,y)=0
\end{cases}
$$  
satisfies $\max_{y\in \R^{N-1}}v(0,y)\leq \theta$.

Due to Proposition \ref{Visole}, there exist $\theta_1>\theta$ such that if an asymptotic profile $v$ solution of \eqref{PAb} verifies $v(0)\geq \theta_1$ then $v\equiv V$.

There is a point $b=b_j$, $-a<b < 0$ such that $v_a(b,0)=\theta_1$. We now translate the solution to center it on $x_1=b$. That is, we let $\vt_a(x_1,y)=v_a(x_1+b,y)$ defined for $x_1\in (-a-b,a-b)$ and $y \in  \R^{N-1}$. The interval $(-a-b,a-b)$ either converges (along a subsequence) to $(-\infty,+\infty)$ or to some $(-d,+\infty)$ with $0\leq d< \infty$. In both cases, by standard elliptic estimates, one can strike out a subsequence of $\vt_a$, denoted again $\vt_a$, such that $\vt_a$ converges locally to some function $w$ where $w$ satisfies:
\begin{equation}
\label{Eqw}
\begin{cases}
-\Delta w + \alpha g(y) w =f(w) \quad \text{in } (-d,+\infty)\times \R^{N-1} \\
\partial_1 w \leq 0, \quad w(0,0)=\theta_1.
\end{cases}
\end{equation}
In case the interval is converging to $(-d,+\infty)$, in addition we know that $w(-d,y)=V(y)$. 
If the interval converges to $\R$, then $\lim_{x\rightarrow -\infty} w(x_1,y)$ exists and is some function $W(y)$ which is then a solution of the profile equation \eqref{PAb}. But since $w(0,0)=\theta_1$, we know that $W(0)\geq \theta_1$. By the definition of $\theta_1$, this implies that $W\equiv V$. Therefore, denoting $d=\infty$ in case $(-a-b,a-b) \rightarrow  \R$, in both cases, we get
$$
\forall y\in \R^{N-1} \quad w(-d,y)=V(y)
$$
where now $0 \leq d\leq +\infty$. We also know that $w(+\infty,y)=\psi(y)$ exists with $0\leq \psi<V$.

Multiply \eqref{Eqw} by $\partial_1 w$ and integrate over $(-d,+\infty) \times \R^{N-1}$ to get
$$
\int_{\{x_1=-d\}} \frac{1}{2}(\partial_1 w)^2+J(\psi)-J(V)=0
$$
where $\partial_1 w=0$ if $d=\infty$. In all cases, we get
$$
J(V)\geq J(\psi)
$$
Since $V$ minimize $J_\alpha$, we obtain $V\equiv \psi$ and $w(x_1,y)=V(y)$ for all $x_1\in (-d,+\infty)$ but this contradicts the renormalization $w(0,0)=\theta_1$.

We have thus reached a contradiction. This shows that for large $a$, $c_a\geq 0$, which completes the proof of the Proposition.   
\end{proof}

Let us now turn to the proof of the existence of traveling front solutions of \eqref{FPb}. Since $c_a$ and $u_a$ are bounded, by standard elliptic estimates, we can strike out a sequence $a=a_j\rightarrow \infty$ (we continue to denote subsequences by $a$) such that $c_a\rightarrow c\geq 0$ and $u_a\rightarrow u$. We know that $(c,u)$ satisfies the equation
$$
- \Delta u +c \partial_1 u +\alpha g(y) u =f(u) \quad \text{ in } \R^N
$$
with $\partial_1 u \leq 0$ and $\max_{R^{N-1}}u(0,\cdot)=\theta$. It remains to identify the limits as $x_1\rightarrow \pm \infty$. These $\lim_{x_1\rightarrow \pm \infty} u(x_1,y)=u_{\pm}(y)$ exist and are solutions of the asymptotic profile equation \eqref{PAb}. Now since $0\leq u_+(y)=\lim_{x_1\rightarrow +\infty} u(x_1,y) \leq \theta$ and all positive solutions $w$ of \eqref{PAb} satisfy $\max w>\theta$, we have $u_+\equiv 0$. 

We claim that $u_- (y)=\lim_{x_1\rightarrow -\infty}u(x_1,y)$ coincides with $V(y)$. Clearly, $0<u_-\leq V$. Argue by contradiction that $u_- \not\equiv V$, implying $u_-<V$. By assumption, $u_-$ is an unstable solution of \eqref{PAb} in the sense that $\lambda_1[u_-]<0$. Let us construct a super-solution of the stationary equation, that is a $w$ with 
$$
-\Delta w+\alpha g(y) w \geq f(w)
$$
such that $w$ is a compact perturbation of $u_-$ and as close as we wish to $u_-$.

Consider the linearized equation about $u_-$:
$$
-\Delta \psi -f'(u_-(y))\psi +\alpha g(y) \psi= \lambda_1[u_-] \psi
$$
with $\lambda_1[u_-]<0$. We know that $\lambda_1[u_-]$ is the limit of the Dirichlet principal eigenvalue in a ball when the radius goes to infinity (This follows from the Rayleigh quotient minimization). Therefore, $R>0$ can be chosen sufficiently large so that the principal eigenvalue $\mu$ and associated eigenfunction $\psi$ of 
$$
\begin{cases}
-\Delta \psi +\alpha g(y)\psi-f'(u_-)\psi=\mu \psi \quad \text{ in } B_R \\
\psi=0 \quad \text{on } \partial B_R, \qquad \psi >0 \quad \text{in } B_R
\end{cases}
$$
satisfy $\mu<0$.

Consider the function $\zeta (x_1,y)=\cos (\omega x_1) \psi(y)$ defined for $x_1\in (-L,L)$ with $L=\frac{\pi}{2 \omega}$ and $|y|<R$. We note $D=(-L,L)\times B_R$. This function is positive and satisfies:
$$
\begin{cases}
-\Delta \zeta +\alpha g(y)\zeta - f'(u_-)\zeta=(\mu+\omega^2)\zeta \quad \text{ in } D \\
\zeta = 0 \quad \text{on } \partial D
\end{cases}
$$
Choose $L$ large enough so that $\mu+\omega^2<0$. Then let $w(x_1,y)=u_-(y)-\ep \zeta(x_1,y)$ with $\ep>0$ and $(x_1,y)\in D$. This function satisfies 
$$
-\Delta w+\alpha g(y) w - f(w)= \left( -(\mu+\omega^2)+ \frac{f(u_-) - f(u_--\ep \zeta)}{\ep \zeta} -f'(u_-)\right) \ep \zeta.
$$
Since $\mu+\omega^2<0$, we can choose $\ep$ sufficiently small so that 
$$
-\Delta w+\alpha g(y) w - f(w)\geq 0 \quad \text{ in } D \text{ and } w>0.
$$
Furthermore, because $\ep \zeta =0$ on $\partial D$ and $\ep \zeta >0$ in $D$, that is $w<u_-$ in $D$, if we extend $w$ by choosing $w(x_1,y)=u_-(y)$ for all $(x_1,y) \not \in D$, we have constructed a (generalized) super-solution of the problem (see e.g. \cite{BL80}). 

Let us now derive a contradiction. We consider two cases.

$\bullet$ Cases (i): Suppose $c>0$. Then $U(t,x_1,y)=u(x_1-ct,y)$ is a solution of the evolution equation
$$
\partial_t U - \Delta U +\alpha g(y) U=f(U) \qquad  t\in \R, \;(x_1,y)\in \R^N.
$$  
Now $U \xrightarrow{t\rightarrow -\infty}0$ locally uniformly in $(x_1,y)$. Furthermore, for all times $U(t,x_1,y)\leq u_-(y)$. Since $w$ is a compact perturbation of $u_-$ for a time $t_0$ sufficiently negative, we get 
$$
\forall (x_1,y)\in \R^N \quad U(t_0,x_1,y)\leq w(x_1,t).
$$ 
Now when $U(t,x_1,y) \xrightarrow{t\rightarrow +\infty} u_-(y)$ locally uniformly  and we get a contradiction since $U(t,x_1,y)\leq w(x_1,y) <u_-(y)$ for all $(x_1,y)\in D$. 

$\bullet$ Case (ii): The case that remains to be studied is $c=0$ (since we already have $c\geq 0$). Then $u(x_1,y)$ is a stationary solution of the same equation that $w$ is a super-solution of. Since $u(-\infty,y)=u_-(y)$ and $u(+\infty,y)=0$, and since $w=u_-$ outside a compact set, after a translation, we can assume that $u_h=u(x_1+h,y)\leq w(x_1,y)$(for large enough $h$). Define 
$$
h^*=\inf \{h\in \R, u(x_1+h,y) \leq w(x_1,y) \text{ in } \R^N \}.
$$
Clearly, $h^*>-\infty$ (for $w<u_-$ at some points). Then $w(x_1,y) \geq u(x_1+h^*,y)=u_{h^*}$ and $\min (w - u_{h^*})=0$ is necessarily achieved at a point of $\overline{D}$. Since $w(x_1,y)=u_-(y)>u(x_1+h,y)$ for all $h$ if $(x_1,y)\not \in D$, we see that the maximum  is achieved at an interior point of $D$. Writing $ w-u_{h^*}\geq 0$ as a super-solution of a linear elliptic equation in $D$, we derive a contradiction with the strong maximum principle. 

Therefore in all cases, the solution $u$ satisfies the limiting condition:
$$
u(-\infty,y)=V(y), \qquad u(+\infty,y)=0.
$$
Therefore $(c,u)$ is a solution of the traveling front equation \eqref{FPb}.


\section{The model of cortical spreading depression}
\label{section: CSD}

We consider here more general versions of the model \eqref{SD} described in the Introduction. The problems studied in this paper have the following general form
\begin{equation}
\label{Egen}
\partial_t u -\Delta u=h(y,u) \quad x=(x_1,y)\in \R^N.
\end{equation}
In the modeling context $N=2$ and $3$ are the cases of interest. As indicated in the Introduction, this equation also describes cortical spreading depressions (CSD). There the wave propagates in a medium composed of two different components, the gray and white matters of the brain, with a narrow transition area separating them. 

Thus we consider in this section functions $h(y,u)$ of the following type:
\begin{equation}
\label{f1}
h(y,u)=f(u) \text{ for } |y|\leq L_1
\end{equation}
\begin{equation}
\label{f2}
 h(y,u)\leq -mu \text{ for } |y|\geq L_2
\end{equation}
\begin{equation}
\label{f3}
h(y,u)+mu\xrightarrow{|y|\rightarrow +\infty} 0 \quad \text{uniformly for } u\in \R^+
\end{equation}
where $0<L_1\leq L_2 <\infty$ and $K\geq m>0$ are given parameters and $f$ is of bistable form. That is we assume that $f$ verifies conditions \eqref{fbistable}-\eqref{fcpos}of section \ref{section: APbist}. Note that in particular, we assume 
$$
\int_0^1 f(s)ds >0.
$$ 
We also assume that $y\mapsto h(y,s)$ is continuous and that $s\mapsto h(y,s)$ is Lipschitz continuous for all $s\in [0,1]$ (and $|y| \not =L_1$ in case $L_1=L_2$). Lastly we assume that
$$
\forall s\in [0,1] \; \forall y\in \R^{N-1} \quad h(y,s)\leq \max \{ f(s), -ms\}.
$$

\subsection{The asymptotic profile equation}

We start as usual with the profile equation
\begin{equation}
\label{APSD}
\begin{cases}
-\Delta V=h(y,V) \quad y\in \R^{N-1}, \\
V\geq 0, \quad V \text{ bounded.}
\end{cases}
\end{equation}
We recall that $\lambda_1[V]$ is the principal eigenvalue of the linearized equation about $V$. This eigenvalue can be defined as 
$$
\lambda_1[V]=\inf_{\varphi\in H^1(\R^{N-1})} \frac{\int |\nabla \varphi |^2 -\partial_2 h (y,V) \varphi^2
}{\int \varphi^2}.
$$
Associated with \eqref{APSD} is the energy functional:
$$
J(w)=\int_{\R^{N-1}} \left( \frac{1}{2} |\nabla w|^2 -H(y,w)\right) dy
$$
where $H(y,z)=\int_0^z h(y,s)ds$. Note that owing to condition \eqref{f2}, $J(w)$ is well defined for all $w\in H^1(\R^{N-1})$.
\begin{theorem}
\label{thmSD}
There exist critical radii $0<L_*\leq L^*<\infty$ with the following properties:
\begin{description}
\item[i)] For $L_2<L_*$, there is no solution other than $0$ to the asymptotic profile equation \eqref{APSD}.
\item[ii)] For $L_1>L^*$ (independently of $L_2$), there exists a maximum solution $V$ of \eqref{APSD} and this solution is stable in the sense that $\lambda_1[V]\geq 0$. 
\item[iii)] For all $L_1>L^*$, the minimum of the energy functional  is achieved at some non-zero function $V_J(y)$, i.e.
$$
J(V_J)=\min_{w\in H^1(\R^{N-1})} J(w) <0.
$$
\end{description}
\end{theorem}
\begin{proof}
i) Since the equation implies that $-\Delta u +mu\leq 0$ for all $|y|\geq L_2$ in $\R^{N-1}$, and $u>0$ is bounded, by Theorem \ref{BR} we know that $u$ and $|\nabla u|$ have exponential decay as $|y|\rightarrow +\infty$. Then we get
$$
\min(1,m) \|u\|_{H^1(\R^{N-1})} \leq \int_{\R^{N-1}} |\nabla u| ^2 +m u^2 \leq \int_{B_{L_2}} \big( f(u) +m u \big)u \leq K \int_{B_{L_2}} u^2.
$$
We know, by Sobolev embedding and H\"{o}lder inequality, that 
$$
\int_{B_{L_2}} u^2 \leq \eta(L_2) \|u\|_{H^1(\R^{N-1})}
$$
where $\eta(L_2)\rightarrow 0$ as $L_2\rightarrow 0$. Therefore for $L_2$ small enough, these inequalities yield $u\equiv 0$.

ii) Next, since $1$ is a super-solution of the equation in $\R^{N-1}$, there exists a maximum solution of equation \eqref{APSD} that we denote $V$. By what we have just seen, $V\equiv 0$ for $L_2$ sufficiently small. Let us now show that $V>0$ for $L_1$ sufficiently large. 

Let us consider the energy restricted to the ball of radius $R\leq L_1$
$$
J_R(w)=\int_{B_R} \left(\frac{1}{2} |\nabla w|^2 -F(w) \right) dy
$$
where $F(z)=\int_0^1 f(s)ds$. We know (see the proof of Theorem \ref{thmPAb}) that for $R$ sufficiently large there exists a minimum $w_R$ of 
$$
J_R(w_R)=\inf_{w\in H^1_0(B_R)} J_R(w) <0.
$$
Then $w_R>0$ in $B_R$ and $w_R$ is solution of $-\Delta w_R=f(w_R)$ in $B_R$ and $w_R=0$ on $\partial B_R$. Extending $w_R$ by $0$ outside the ball $B_R$, we get a global (generalized) sub-solution. The solution $V$ is such that $V\geq w_R$ (since $V$ is the maximum solution). This implies that $V\not\equiv 0$ and therefore $V>0$ in $\R^{N-1}$ for $L_1 \geq R$.

iii) Now for $L_1\geq R$, clearly
$$
\inf_{w\in H^1(\R^{N-1})} J(w)\leq J_R(w_R)<0.
$$
Let us now show that the infimum is achieved.

Let $(w_n)$ be a minimizing sequence: $J(w_n)\rightarrow \inf J <0$ for $n\rightarrow +\infty$. Note that $J$ is bounded from below.
Writing
\begin{eqnarray*}
J(w_n)& \geq&  \frac{1}{2} \int_{ \R^{N-1}\setminus B_{L_2}}|\nabla w_n|^2+m {w_n}^2 + \int_{B_{L_2}}\frac{1}{2}|\nabla w_n|^2 -H(y,w_n) \; \\
& \geq &   \frac{1}{
2}\int_{ \R^{N-1}\setminus B_{L_2}}|\nabla w_n|^2+m {w_n}^2 +\int_{B_{L_2}}\frac{1}{2}|\nabla w_n|^2 -C+\ep {w_n}^2 \\
& =& -C|B_{L_2}|+\tilde{\ep}\|w_n\|_{H^1(\R^{N-1})}
\end{eqnarray*}
we can strike out a subsequence still denoted $(w_n)$ such that $w_n \rightarrow w$ weakly in $H^1(\R^{N-1})$. Now using \eqref{f3}, for every $\ep>0$, there exists $R=R(\ep)>0$ such that $\left| F(y,s)+\frac{m}{2}s^2 \right|<\ep s^2$ for all $|y|\geq R$ and all $s\in \R^+$ (there is no loss in generality in assuming $w_n\geq 0$ as ${w_n}^+$ is also a minimizing sequence). Therefore 
$$
J(w_n)=\frac{1}{2}\int_{ \R^{N-1}}|\nabla w_n|^2+\frac{m}{2} \int_{|y|\geq R} {w_n}^2 +r(\ep) - \int_{|y|\leq R} H(w_n,y)
$$
where $|r(\ep)|\leq C\ep$ for some constant $C>0$.

By compact injection of $H^1(\R^{N-1})\hookrightarrow L^2(B_R)$, we can assumme that $w_n \rightarrow w$ strongly in $L^2(B_{L_2})$. Then by standard arguments relying on Lebesgue's dominated convergence Theorem, we see that
$$
\int_{B_R}H(y,w_n)\rightarrow \int_{B_R} H(y,w).
$$
Owing to Sobolev embedding and H\:{o}lder inequality, it is straightforward to check that the quantity
$$
\frac{1}{2}\int_{ \R^{N-1}}|\nabla w_n|^2+\frac{m}{2} \int_{|y|\geq R} {w_n}^2
$$ 
defines the square of a norm equivalent to the usual $H^1(\R^{N-1})$ norm. Hence using the lower semi-continuity of the norm, we get
$$
\lim_{n\rightarrow \infty} J(w_n)=\inf J\geq \frac{1}{2}\int_{ \R^{N-1}}|\nabla w|^2+\frac{m}{2} \int_{|y|\geq R} {w}^2+r(\ep)-\int_{B_R} H(y,w).
$$
Now using again \eqref{f3}, we get:
$$
J(w)=\int_{ \R^{N-1}}|\nabla w|^2-\int_{\R^{N-1}} H(y,w) \leq \inf_{H^1(\R^{N-1})} J+2 |r(\ep)|.
$$
Since $\ep>0$ is arbitrarily small we get
$\displaystyle
J(w)=\inf_{ H^1 (\R^{N-1})} J.
$
\end{proof}
\begin{remark}
By using the method of \cite{BL80bis}, one can show that for $L_1$ large enough there exists a second solution in $\R^{N-1}$.
\end{remark}

We will now make use of the condition that the stable solution of \eqref{APSD} is unique. In the paper of Chapuisat and Joly \cite{CJ11}, it is argued by phase plane method, that for the case $N-1=2$, $L_1=L_2$ and $h(y,s)=-ms$ for $|y|\geq L_2$ that indeed this is the case. We note that it is an interesting open problem to derive such uniqueness results in more general situations or to complete the heuristic part of the argument of \cite{CJ11}.

\subsection{Traveling fronts for the CSD model}

In this section, we prove Theorem \ref{thmSDTF}. The proof is similar as in Section \ref{section: TFbist}. There we used that $h(y,u)= f(u)-\alpha g(y) u$ with $g\rightarrow +\infty$. But actually, the same properties that were entailed one can derived for $h(y,u)\leq -mu$ for large $|y|$. We start by constructing a solution of
\begin{equation}
\label{TFSDa}
\begin{cases}
\Delta u_a-c_a \partial_1 u_a =h(y,u_a) \quad \text{ in } (-a,a)\times \R^{N-1} \\
u_a(-a,y)=V(y), \quad u_a(+a,y)=0, 
\end{cases}
\end{equation}
that verifies
\begin{equation}
\label{normSD}
\sup_{y\in \R^{N-1}} u_a(0,y)=\theta
\end{equation}
for $a\geq 1$ and where $\theta$ is the unstable 0 of $f(u)$, that is $f(\theta)=0$ and $0<\theta<1$. We recall that $c_a$ is uniquely determined by the renormalization condition \eqref{normSD}. 

\hspace{-7mm}
\begin{minipage}{0.7\textwidth}
Let $\tf(u)=\max \{f(u), -m(u)\}$. Note that $\tf$ itself is bistable:
$$
\tf(0)=\tf(\theta)=\tf(1), \quad \tf(s)<0 \text{ in }(0,\theta), \quad \tf(s)>0 \text{ in }(\theta, 1).
$$
\end{minipage} \hfill 
\begin{minipage}{0.28\textwidth}
\includegraphics[width=5cm]{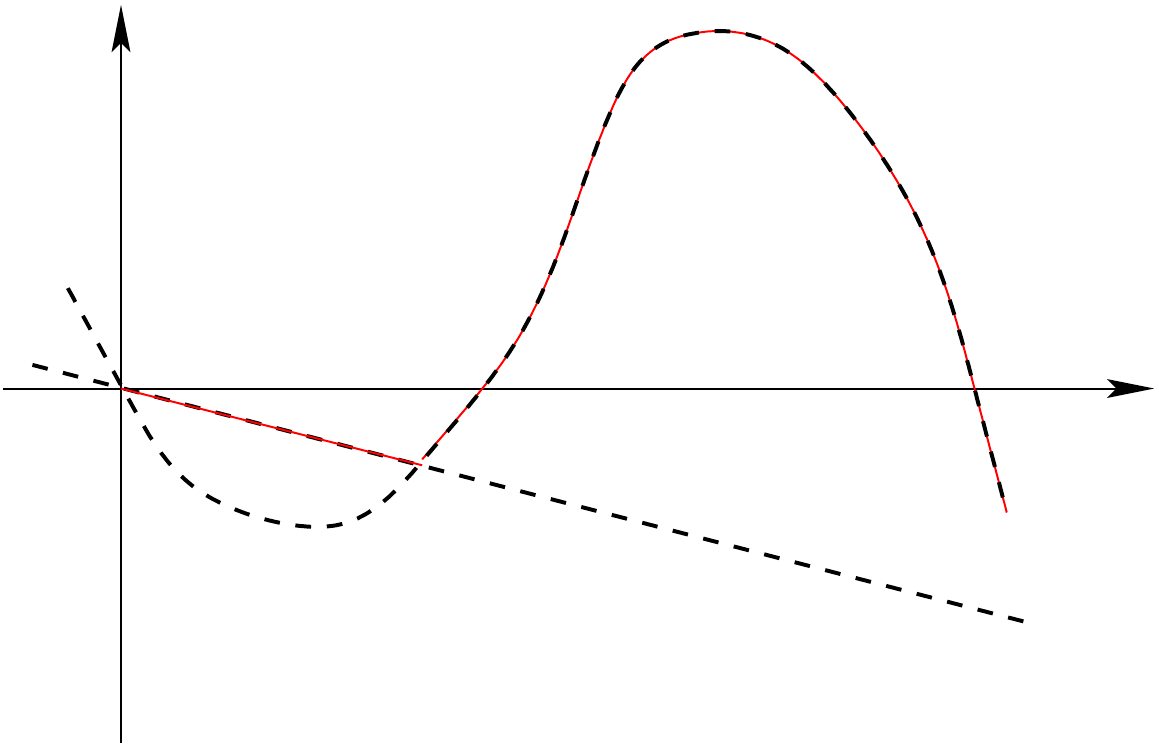}
\end{minipage}
We denote by $z_a^c$ the solution of 
\begin{equation}
\label{1DSD}
\begin{cases}
-z''-cz'=\tf(z) \\
z(-a)=1, \quad z(+a)=0.
\end{cases}
\end{equation}
The function $z_a^{c_a}$ is a supersolution of \eqref{TFSDa} thus $u_a \leq z_a^{c_a}$. In view of \eqref{normSD} this implies that $z_a^{c_a}(0)\geq \theta$ and this implies that $c_a\leq \gamma_a$ where $\gamma_a$ is the unique value of $c$ such that the solution of \eqref{1DSD} verifies $z_a^{\gamma_a}(0)=\theta$. This as before yields the upper bound for $c_a$.

The lower bound is achieved in the same manner as in the section \ref{section: TFbist} and the convergence for $a\rightarrow +\infty$ also.

\section*{Acknowledgments}

This study was funded by the European Research Council under the European Union's Seventh Framework Programme (FP/2007-2013)/ERC Grant Agreement n.321186 - ReaDi - Reaction-Diffusion Equations, Propagation and Modelling. Henri Berestycki was also partially supported by an NSF FRG grant DMS-1065979. Part of this work was carried out while he was visiting the Department of Mathematics at the University of Chicago. 

We wish to thank Laurent Desvillettes for bringing to our attention the population genetics model described in equations \eqref{DP1}-\eqref{DP2}.


\bibliographystyle{plain}
\bibliography{biblio}	

\end{document}